\documentclass[10pt]{amsart}
\usepackage{graphicx}
\usepackage{amscd}
\usepackage{amsmath}
\usepackage{amsthm}
\usepackage{amsfonts}
\usepackage{amssymb}
\usepackage{mathrsfs}
\usepackage{enumerate}
\usepackage{amsrefs}
\usepackage{xcolor}
\usepackage[colorlinks, citecolor=blue, linkcolor=red, pdfstartview=FitB]{hyperref}
\usepackage[latin1]{inputenc}
\usepackage{tikz-cd}

\newcommand{\bB}{{\mathbb{B}}}
\newcommand{\bC}{{\mathbb{C}}}
\newcommand{\bD}{{\mathbb{D}}}

\newcommand{\bM}{{\mathbb{M}}}

\newcommand{\bR}{{\mathbb{R}}}

\newcommand{\bT}{{\mathbb{T}}}

  \newcommand{\A}{{\mathcal{A}}}
  \newcommand{\B}{{\mathcal{B}}}
  
  \newcommand{\D}{{\mathcal{D}}}

\renewcommand{\L}{{\mathcal{L}}}

\renewcommand{\P}{{\mathcal{P}}}
  \newcommand{\Q}{{\mathcal{Q}}}
  
\renewcommand{\S}{{\mathcal{S}}}

  \newcommand{\W}{{\mathcal{W}}}

  \newcommand{\Z}{{\mathcal{Z}}}

\newcommand{\fK}{{\mathfrak{K}}}

\newcommand{\fl}{{\mathfrak{l}}}

\newcommand{\fT}{{\mathfrak{T}}}

\newcommand{\rA}{\mathrm{A}}

\newcommand{\rC}{\mathrm{C}}

\newcommand{\eps}{\varepsilon}
\renewcommand{\phi}{\varphi}

\newcommand{\upchi}{{\raise.35ex\hbox{$\chi$}}}

\newcommand{\ol}{\overline}

\newcommand{\qand}{\quad\text{and}\quad}

\newcommand{\re}{\operatorname{Re}}
\newcommand{\im}{\operatorname{Im}}
\newcommand{\spn}{\operatorname{span}}

\newcommand{\spec}{\operatorname{spec}}

\newcommand{\nr}{\operatorname{nr}}

\newtheorem{lemma}{Lemma}[section]
\newtheorem{theorem}[lemma]{Theorem}
\newtheorem{proposition}[lemma]{Proposition}
\newtheorem{corollary}[lemma]{Corollary}

\newtheorem{theoremx}{Theorem}

\theoremstyle{definition}
\newtheorem{example}[lemma]{Example}
\begin{document}

\author{Rapha\"el Clou\^atre}
\address{Department of Mathematics, University of Manitoba, Winnipeg, Manitoba, Canada R3T 2N2}
\email{raphael.clouatre@umanitoba.ca\vspace{-2ex}}
\thanks{The author was partially supported by an NSERC Discovery Grant.}
\begin{abstract}
Through the lens of noncommutative function theory, we study restrictions of pure states to  unital subspaces of $\rC^*$-algebras, in the spirit of the Kadison--Singer question. More precisely, given a unital subspace $M$ of a $\rC^*$-algebra $B$, the fundamental problem is to describe those pure states $\omega$ on $B$ for which  $E_\omega=\{\omega\}$, where $E_\omega$ is the set of states on $B$ extending $\omega|_M$. In other words, we aim to understand when $\omega|_M$ admits a unique extension to a state on $B$.  We find that the obvious necessary condition that $\omega|_M$ also be pure is sufficient in some naturally occurring examples, but not in general. Guided by classical results for spaces of continuous functions, we then turn to noncommutative peaking phenomena, and to the several variations on noncommutative peak points that have previously appeared in the literature. We  perform a thorough analysis of these various notions, illustrating that all of them are in fact distinct, addressing their existence and, in some cases, their relative abundance. Notably, we find that none of the pre-existing notions provide a solution to our main problem. We are thus naturally led to introduce a new type of peaking behaviour for $\omega$, namely that the set $E_\omega$ be what we call an $M$-\emph{pinnacle set}. Roughly speaking, our main result is that $\omega|_M$ admits a unique extension to $B$ if and only if $E_\omega$ is an $M$-pinnacle set. 
\end{abstract}

\title{Restrictions of Pure states to subspaces of $\rC^*$-algebras}

\maketitle

\section{Introduction}
The typical setup for this paper is that of a unital $\rC^*$-algebra $B$ and a unital  subspace $M\subset B$. A unital contractive linear functional $\phi:M\to\bC$ is called a \emph{state}. We say that $\phi$ is a \emph{pure} state when it is an extreme point of the convex set $\S(M)$ of all states on $M$. In this work we are concerned with the following fundamental problem.

\vspace{3mm}

\textbf{Main question.} Let $\omega$ be a pure state on $B$ and let $E_\omega$ denote the set of states on $B$ extending $\omega|_M$. When does $E_\omega$ reduce to the singleton $\{\omega\}$?
\vspace{3mm}

This can be viewed as a generalized version of a question that has garnered sustained interest in the case where $M$ is also a $\rC^*$-algebra \cite{anderson1979},\cite{archbold1980},\cite{ABG1982},\cite{AAP1986},\cite{archbold1999},\cite{archbold2001}. Notably, when $M$ is a discrete maximal abelian self-adjoint subalgebra of the $\rC^*$-algebra $B(H)$ of bounded linear operators on a Hilbert space $H$, the issue of uniqueness of extensions for states gave rise to the celebrated Kadison--Singer problem \cite{KS1959}, which was eventually resolved positively decades later \cite{MSS2015}. 

The starting point of this discussion is the simple observation that, in order for $E_\omega$ to be a singleton, it is necessary that $\omega|_M$ be a pure state on $M$. In this case, a standard convexity argument reveals that $E_\omega$ is a singleton precisely when $\omega$ is the only \emph{pure} state in $E_\omega$. Therefore, it is natural to define $M$ to have the \emph{pure extension property} in $B$ if, for every pure state $\omega$ on $B$, we have
\[
E_\omega=\{\omega\} \quad \text{if and only if } \omega|_M \text{ is pure}.
\]
When the $\rC^*$-algebra $B$ is commutative and generated by $M$ (i.e. $B=\rC^*(M))$, then because all pure states on $B$ are multiplicative, it readily follows that $M$ necessarily has the pure extension property, thereby answering our main question.

In Section \ref{S:pureuep}, we explore the pure extension property for unital subspaces of general $\rC^*$-algebras. Despite it being a natural property, it does not seem to have been investigated in detail before (although it is alluded to in \cite{BN2012}). In Proposition \ref{P:freespec}, we exhibit examples of subspaces with the pure extension property coming from the theory of matrix convexity and free spectrahedra. 
We also illustrate in Example \ref{E:Popescu}  how natural subspaces of multipliers on both the full and the symmetric Fock space have the pure extension property. Further, we  show in Theorem \ref{T:M2} that any unital subspace of  $\bM_2$ has this property. In general however, the pure extension property is known to fail. Indeed, there is a unital subspace of the $4\times 4$ complex matrices $\bM_4$ that fails to have this property \cite{sherman2023}.

In light of this, we must look elsewhere to answer our main question in general. A  contractive net $(e_\lambda)$ in $B$ is an \emph{excision} for the pure state $\omega$ on $B$ if $\lim_\lambda \omega(e_\lambda)=1$ and
\[
\lim_\lambda \|e^*_\lambda (b-\omega(b)I)e_\lambda\|=0, \quad b\in B.
\]
Excisions were first introduced and utilized by Anderson in \cite{anderson1979} (albeit under a different name), where our main question was fully answered in the special case where $M$ is a unital $\rC^*$-subalgebra of $B$: $E_\omega=\{\omega\}$ if and only if $\omega$ admits an excision in $M$. 
Subsequently, the idea of excision was fleshed out further in \cite{AAP1986}. Notable subsequent applications include the study of quasidiagonal $\rC^*$-algebras \cite{brown2005}, as well as some considerations surrounding Arveson's hyperrigidity conjecture \cite{clouatre2018lochyp} (once again,  different terminology was used in this last paper). Related concepts also appear in \cite{BW2017}. 

In Section \ref{S:excision}, we examine the existence of excisions in $M$ for general unital subspaces, and relate it to Hay's noncommutative peaking phenomenon for projections \cite{hay2007}. There is a unique  projection $\fl_\omega\in B^{**}$ such that 
 \[
 \{\xi\in B^{**}:\omega(\xi^*\xi)=0\}=B^{**}(I-\fl_\omega).
 \]
Building on deep work of Blecher, Hay, Neal and Read, we show in Corollary \ref{C:excisionequiv} that an excision for $\omega$ exists in a (possibly non-selfadjoint) unital, separable, norm-closed subalgebra $A\subset B$ if and only if $\fl_\omega$ is a so-called \emph{$A$-peak projection}. This latter peak property is characterized in Proposition \ref{P:peakprojnonorthog}, which in particular makes it clear that it coincides with the variant considered in \cite{clouatre2018lochyp}. Another easy consequence is  that $E_\omega=\{\omega\}$ whenever $\fl_\omega$ is a peak projection. As we illustrate in Example \ref{E:ueppeakproj}, the converse is typically false, so excisions do not yield a full answer to our question in the general case.

In spite of this, by analogy with the commutative case, one expects that alternative peaking phenomena may still be relevant. Let us recount some of the classical details that guide us throughout the paper.

Let $X$ be a compact Hausdorff space and let $A\subset \rC(X)$ be a unital  norm-closed  subalgebra  that separates the points of $X$  (equivalently, $\rC^*(A)=\rC(X)$).  Then, the pure state
 $\omega$ is given by evaluation at some point $x\in X$, and the following statements are equivalent. 
\begin{enumerate}[{\rm (a)}]
\item $E_\omega=\{\omega\}$
\item $\omega|_A$ is pure.
\item $x$ is an \emph{$A$-peak point}: there is a function $a\in A$ with $\|a\|=1$ such that $\{x\}=\{y\in X: |a(y)|=1\}$.
\end{enumerate}
The equivalence of (a) and (b) was mentioned above, while the implication (c)$\Rightarrow$(a) is straightforward.  The crucial implication is (a)$\Rightarrow$(c), which relies on the following.

\vspace{3mm}
\textbf{Bishop's criterion.} 
 There is $0<\eps<1$ for which, given any compact subset $K\subset X$ not containing $x$, we can find $a\in A$ such that $a(x)=1, \|a\|\leq 1+\eps$ and $|a(y)|<\eps$ for every $y\in K$.
\vspace{3mm}

 A clever trick of Bishop can be used to show that, for every unital norm-closed subspace $M\subset\rC(X)$, whenever the singleton $\{x\}$ is a $G_\delta$ subset of $X$, then $x$ is an $M$-peak point if it satisfies this criterion \cite[Theorem 6.5]{BdL1959} (see also \cite[Theorem II.11.1]{gamelin1969}). The property of $\{x\}$ being a $G_\delta$ can be removed at the cost of obtaining the weaker conclusion that $x$ be a so-called weak peak point \cite[Lemma II.12.1]{gamelin1969}. A generalized criterion for peak sets was formulated in \cite{jarosz1984}, while noncommutative versions of Bishop's criterion have been utilized in \cite{BR2011}. 

Going back to the general setting where $B$ may not be commutative and $M\subset B$ is a unital subspace, the aforementioned equivalence of (a) and (c) suggests that an answer to our question might yet  be found in some form of noncommutative peaking. On the other hand, we saw above that the property that $\fl_\omega$ be an $M$-peak projection is typically too strong to fully characterize the unique extension property $E_\omega=\{\omega\}$. 

We are thus led to consider relaxations of this notion, a task that we undertake in Section \ref{S:ncpeak}. Therein, we examine what we call \emph{$M$-peak sets/states} and \emph{$\Q(M)$-peak sets/states};  throughout, we put $\Q(M)=\{x^*x:x\in M\}$. Similar concepts have arisen in recent years in investigations related to Arveson's hyperrigidity conjecture \cite{arveson2011},\cite{clouatre2018unp},\cite{clouatre2018lochyp} and other structural operator algebraic questions \cite{CTh2022},\cite{CTh2023}. All three variants of noncommutative peak points collapse to the usual notion of peak point when $B$ is commutative.

We give a characterization in Proposition \ref{P:peakset} of those states $\omega$ such that $E_\omega$ contains a non-empty $M$-peak or $\Q(M)$-peak set. We emphasize here that such noncommutative peak sets are much more restrictive than their usual classical counterparts, being contained entirely in $E_\omega$. Recall indeed that, in the commutative picture, $E_{\omega}$ is a singleton in this case, so this notion recovers that of a peak point, as opposed to the weaker notion of a peak set. In  Example \ref{E:Dirichlet}, we study in detail a non-trivial concrete example of an operator algebra acting on the classical Dirichlet space of holomorphic functions, and analyze the pure states on this algebra from the point of view of peaking.

In Section \ref{S:abundance}, we further substantiate that the states $\omega$ for which $E_\omega$ contains a non-empty $M$-peak or  $\Q(M)$-peak set provide a meaningful relaxation of $M$-peak projections. This is accomplished upon exploring the relative abundance of these states. As noted already in \cite{hay2007}, peak projections may fail to exist in even the simplest noncommutative cases. By contrast, when $B$ is commutative, $M$-peak points are always dense in the set of pure states on $B$ that restrict to be pure on $M$. Hence, one may hope that our relaxed notions may also always be somewhat plentiful. This is indeed the case, as we establish in Theorem \ref{T:P0dense} and Corollary \ref{C:Qpeaknorm}, which we summarize below.

\begin{theoremx}\label{T:A}
Let $B$ be a unital  separable $\rC^*$-algebra. Then, the following statements hold.
\begin{enumerate}[{\rm (i)}]

\item  Let $M\subset B$ be a unital norm-closed subspace. Let $\P(M)$ denote the set of pure states $\omega$ on $B$ such that $E_\omega$ contains a non-empty $M$-peak set.  Then, the set $ \{\omega|_M:\omega\in \P(M)\} $ is weak-$*$ dense in the pure states on $M$.  

\item  Let $A\subset B$ be a unital norm-closed subalgebra. Let $\P(\Q(A))$ denote the set of pure states $\omega$ on $B$ such that $E_\omega$ contains a non-empty $\Q(A)$-peak set. Then, 
$
\sup_{\omega\in \P(\Q(A))}\omega(a^*a)=\|a\|^2
$
for every $a\in A$.
\end{enumerate} 
\end{theoremx}

While statement (i) above essentially follows by adapting Mazur's classical smooth point theorem (see Theorem \ref{T:Mazur}), statement (ii) requires us to establish a new variant of that fact (Theorem \ref{T:Qsmoothdense}), which may be of independent interest. At the time of this writing, we do not know if (ii) is valid for subspaces instead of subalgebras. In some small matrix algebras, we explain in Corollary \ref{C:M23dense} how the $\Q(A)$-peak sets in Theorem \ref{T:A} may be chosen to in fact be $\Q(A)$-peak states.

In Section \ref{S:peakuep}, we return to our main question and thoroughly examine the relationships between our relaxed notions of peak sets and the property that $E_\omega=\{\omega\}$. Our conclusions are summarized as follows, for a unital subspace $M\subset B$.
\[
\begin{tikzcd}[arrows=Rightarrow]
 &  \substack{\fl_\omega\text{  is}\\ M\text{-peak projection}} \arrow[dl, shift left=1.5] \arrow[dr,shift left=1.5] &\\
\substack{\omega \text{  is}\\ M\text{-peak state}}\arrow[dr, shift left=1.5] \arrow[rr,"/"{anchor=center,sloped}, shift right=1.5] \arrow[ur,"/"{anchor=center,sloped}, shift left=1.5]& & \substack{\omega \text{  is}\\ \Q(M)\text{-peak state}}\arrow[dl, "/"{anchor=center,sloped}, shift left=1.5]\arrow[ll,"/"{anchor=center,sloped}, shift right=1.5]\arrow[ul,"/"{anchor=center,sloped}, shift left=1.5]\\
 & \substack{E_{\omega}=\{\omega\}}\arrow[ul,"/"{anchor=center,sloped}, shift left=1.5]\arrow[ur,"/"{anchor=center,sloped}, shift left=1.5]&
\end{tikzcd}
\]
Notably, the negated implications fail even for norm-closed subalgebras: we construct a non-trivial counterexample in Example \ref{E:ueppeak2} based on a useful trick allowing one to pass from subspaces to subalgebras (Theorem \ref{T:uepcorner}). 

In Section \ref{S:pinndet}, we introduce the following key definition, inspired by Bishop's criterion above. Let $B$ be a unital $\rC^*$-algebra and let $M\subset B$ be a unital subspace. Let $\omega$ be a state on $B$ and let $\eps>0$. We say that $E_\omega$ is an \emph{$(M,\eps)$-pinnacle set}  if,  given any compact subset $K\subset \S(B)\setminus\{\omega\}$, we can find $x\in M$ with  $\|x\|\leq 1+\eps$ such that $\Omega(x^*x)>1>\psi(a^*a)$ for every $\Omega\in E_\omega$ and $\psi\in  K$. 

Let us give some intuition regarding this definition. Assume that the restriction $\omega|_M$ admits a unique state extension to $B$, namely $\omega$ itself. This means that, given another state $\psi$ on $B$ distinct from $\omega$, there must exist $x\in M$ such that $\omega(x)\neq \psi(x)$. Of course, such an element $x$ typically depends on $\psi$. When $E_\omega$ is an $(M,\eps)$-pinnacle set, given a compact subset $K$ not containing $\omega$, we may find a single element of the form $x^*x$, with $x\in M$ of norm at most $1+\eps$, that simultaneously separates $E_\omega$ from $K$.

The main result of the paper is Theorem \ref{T:ueppeak}, which answers our main question for subalgebras, and reads as follows.

\begin{theoremx}\label{T:B}
Let $A\subset B$ be a unital norm-closed subalgebra and let $\omega$ be a pure state on $B$. Then, the following statements are equivalent.
\begin{enumerate}[{\rm (i)}]
\item The restriction $\omega|_A$ admits a unique state extension to $B$, i.e. $E_\omega=\{\omega\}$.
\item $E_\omega$ is an $(A,\eps)$-pinnacle set for every $\eps>0$.
\item $E_\omega$ is an $(A,\eps)$-pinnacle set for some $\eps>0$.
\item For each compact subset $K\subset \S(B)$ not containing $\omega$,  there is $x\in A$ such that
\[
\Omega(x^*x) >1>\psi(x^*x) 
\]
for $\psi\in K$ and $\Omega\in E_\omega$.
\end{enumerate}
\end{theoremx}

Leveraging once again Theorem \ref{T:uepcorner}, we give in Corollary \ref{C:ueppeakspace} a version of this result that is valid for any unital norm-closed subspace $M\subset B$.

For the sake of comparison with Bishop's criterion, let us consider a closely related notion, which is perhaps more familiar looking. We say that $\omega$ is an \emph{$(M,\eps)$-pinnacle state} if,  given any compact subset $K\subset \S(B)\setminus\{\omega\}$, we can find $x\in M$ with $\|x\|\leq 1+\eps$ such that $\omega(x^*x)>1>\psi(a^*a)$ for every $\psi\in  K$. While reminiscent of Bishop's criterion, this condition is a priori weaker, as $\psi(a^*a)$ is not required to be much smaller than $1$. In addition, this condition is clearly  implied by $E_\omega$ being an $(M,\eps)$-pinnacle set, and also by $\omega$ being a $\Q(M)$-peak point. However, neither implication can be reversed in general, making pinnacle states a bit mysterious. In spite of this, these states can be shown to still enjoy some amount of rigidity (Theorem \ref{T:detbdry}), and do coincide with usual peak points when $B$ is commutative (Corollary \ref{C:pinncomm}).

Finally, we mention that there is a different natural framework for tackling the main question, where noncommutative peak points are taken to be certain special $*$-representations, as opposed to states or projections. Indeed, following Arveson's seminal work  \cite{arveson1969}, it has become a standard tool in modern operator algebras to view a distinguished class of  irreducible $*$-representations of $B$, the so-called boundary representations,  as the noncommutative analogue of the Choquet boundary. This approach has been very successful, especially given its intimate relation with the $\rC^*$-envelope of $M$ see \cite{dritschel2005},\cite{arveson2008},\cite{DK2015}. Just as we did for states, it makes sense in this case to isolate peaking behaviour. This idea was initiated in \cite{arveson2010}, and further pursued in  \cite{CTh2022},\cite{DP2022}.  Connections between peak states and peak representations were examined in \cite[Section 2]{CTh2023}.  

In Section \ref{S:rep}, we explain to which extent this alternative perspective is useful in understanding the unique extension property for states. Indeed, under an additional assumption, Theorem \ref{T:convex} shows how the condition $E_\omega=\{\omega\}$ yields a conclusion about the GNS representation of $\omega$ that is a much closer analogue of Bishop's criterion. Direct analogues of Bishop's criterion in the case where $\fl_\omega$ is a peak projection have appeared previously, see for instance \cite{BR2011} and the references therein.

\vspace{3mm}

\textbf{Acknowledgements.} We are grateful to Jashan Bal, David Blecher, Colin Krisko, David Sherman and Ian Thompson for  stimulating conversations and questions.

\vspace{3mm}

\textbf{Data availability statement.} There is no data associated with this work.

\section{The pure extension property}\label{S:pureuep}

Let $B$ be a unital $\rC^*$-algebra and let $M\subset B$ be a unital subspace. Recall that we say that $M$ has the \emph{pure extension property} in $B$ if every pure state on $M$ admits a unique extension to a state on $B$. A standard argument using  the Krein--Milman theorem reveals that this is the same as requiring that every pure state on $M$ admits a unique extension to a \emph{pure} state on $B$.


As explained in the introduction of the paper, when the $\rC^*$-algebra $B$ is commutative and $M$ generates it, then $M$ automatically has this property, since pure states on $B$ are multiplicative. Of course, for more general $\rC^*$-algebras, pure states are not always multiplicative, and thus the pure extension property  imposes significant restrictions on the pair $M\subset B$. In particular, one might expect that the pure extension property is somewhat rare. It is certainly the case that the property is not automatic:  David Sherman has found  a $2$-dimensional operator system in $\bM_4$  for which this property fails \cite{sherman2023}. (Here and throughout, for a natural number $n\geq 1$ we denote by $\bM_n$ the space of $n\times n$ complex matrices.)  Another example appears in \cite[Example 2.3]{DH2023}. 
Nevertheless, we can exhibit some naturally occurring examples where the property does hold.

\subsection{Spectrahedra and the  pure extension property}

For our first class of examples, we need to recall some definitions.

Let $S$ be an operator system. For each $n\geq 1$, let $\W_n(S)$ denote the set of all unital completely positive maps $\psi:S\to\bM_n$. The sequence $\W(S)=(\W_n(S))$ is then called the \emph{matrix state space} of $S$. In the language of \cite{EW1997}, this is a matrix convex set in the dual space of $S$.

Next, fix integers $d\geq 1$ and $g\geq 1$ along with self-adjoint matrices $A_1,\ldots,A_g\in \bM_{d}$. Put $A=(A_1,\ldots,A_g)$.  For each natural number $n\geq 1$, we define the set $\D_A(n)\subset (\bM_n)^g$ to consist of those $g$-tuples $(X_1,\ldots,X_g)$ of self-adjoint matrices satisfying 
\[
\sum_{k=1}^g A_k\otimes X_k\leq I.
\]
The sequence $\D_A=(\D_A(n))$ is then called the \emph{free spectrahedron} corresponding to $A$. It is a matrix convex set in $\bC^g$. For more detail on free spectrahedra and their relation to matrix convexity, the reader may consult   \cite{kriel2019} and the references therein.

Given $X=[x_{ij}]\in \bM_n$, we let $X^\dagger=[\ol{x_{ij}}]\in \bM_n$. Accordingly, we put 
\[
\D_A(n)^\dagger=\{(X_1^\dagger,\ldots,X_g^\dagger): (X_1,\ldots,X_g)\in \D_A(n)\}.
\]

\begin{proposition}\label{P:freespec}
Let $B$ be a unital $\rC^*$-algebra and let $S\subset B$ be an operator system with $\rC^*(S)=B$. Assume that the matrix state space of $S$ is a free spectrahedron $\D_A$ satisfying $\D_A(2)^\dagger=\D_A(2)$. Then, any pure state on $S$ admits a unique extension to a state on $B$, and this extension is a $*$-representation. In particular, $S$ has the pure extension property in $B$.
\end{proposition}
\begin{proof}
Let $\omega$ be a pure state on $S$. Equivalently, this means that $\omega$ is an extreme point of $\W_S(1)$.
Next,  we may apply \cite[Theorem 6.1]{EHKM2018} (see also the corrected version \cite{EHKM2018CORR}) to conclude that $\omega$ is a so-called Arveson extreme point of $\W(S)$, which implies the desired conclusion.
\end{proof}

Operator systems satisfying the conditions above can easily be constructed as follows. Start with a free spectrahedron $\D_A$ such that $\D_A(2)^\dagger=\D_A(2)$. Then, one can construct the operator system of continuous matrix affine functions on $\D_A$, which will have $\D_A$ as its matrix state space; see \cite[Section 3]{WW1999} for details.

\subsection{The joint numerical range}

For our next class of examples, we need to introduce some notation.
Given finitely many elements $b_1,\dots,b_n$ in a unital $\rC^*$-algebra $B$, we denote by $W(b_1,\ldots,b_n)\subset \bC^n$ the convex subset of those points of the form $(\psi(b_1),\ldots,\psi(b_n))$ for some state $\psi$ on $B$. 
It is easily verified that, given a state $\psi$ on $B$, the point $(\psi(b_1),\ldots,\psi(b_n))\in \bC^n$ is an extreme point of $W(b_1,\ldots,b_n)$ precisely when $\psi$ restricts to a pure state on the unital subspace of $B$ generated by $b_1,\ldots,b_n$.

We now show that certain ``maximal" extreme points of the joint numerical range give rise to pure states with unique extensions. The following standard argument is well known to experts, but we record it as we lack an exact reference.

\begin{proposition}\label{P:WArv}
Let $B$ be a unital $\rC^*$-algebra and let $b_1,\ldots,b_n\in B$ be elements satisfying $\sum_{k=1}^n b_k b_k^*\leq I$. Assume that the extreme points of $W(b_1,\ldots,b_n)$ lie on the unit sphere in $\bC^n$. 
Let $M=\spn\{I,b_1,\ldots,b_n\}$,  let $A\subset B$ denote the algebra generated by $M$ and let $D\subset B$ denote the operator system  generated by $\{ab^*:a,b\in A\}$. If $\omega$ is a pure state on $M$, then there is a unique state $\psi$ on $D$ extending $M$. In fact, the state $\psi$ satisfies
\[
\psi(b_{k}d)=\omega(b_{k})\psi(d) \qand 
\psi(b_k d b_j^*)=\omega(b_{k})\psi(d)\omega(b_j)^*
\]
for each  $1\leq j,k\leq n$ and $d\in D$.
\end{proposition}
\begin{proof}
Let $\psi$ be any state extension of $\omega$ to $D$. 
Since $\omega$ is pure, we see that $(\omega(b_1),\ldots,\omega(b_n))$ is an extreme point of $W(b_1,\ldots,b_n)$, and hence it lies on the unit sphere of $\bC^n$. By the Schwarz inequality we find
\begin{align*}
1&=\sum_{k=1}^n |\omega(b_k)|^2\leq \sum_{k=1}^n \psi(b_kb_k^*)=\psi\left(\sum_{k=1}^n b_k b_k^* \right)\leq 1
\end{align*}
so that 
\[
\psi(b_k b_k^*)=|\omega(b_k)|^2=|\psi(b_k)|^2, \quad 1\leq k\leq n.
\]
Applying \cite[Theorem 3.18]{paulsen2002},  given $1\leq j,k\leq n$ and $d\in D$ we find
\begin{align*}
\psi(b_{k}d)&=\psi(b_{k})\psi(d)=\omega(b_k)\psi(d)
\end{align*}
and
\[
\psi(b_k d b_j^*)=\psi(b_{k})\psi(d)\psi(b_j)^*=\omega(b_{k})\psi(d)\omega(b_j)^*
\]
In particular, $\psi$ is  uniquely determined by $\omega$. 
\end{proof}

The previous result should be compared with \cite[Theorem 3.1.2]{arveson1969},\cite[Corollary 3]{arveson1998},\cite[Theorem 3.4]{CH2018}. As an application, we show how certain spaces arising naturally in multivariate operator theory have the pure extension property.

\begin{example}\label{E:Popescu}
Let $d\geq 1$ be an integer and let $F^2_d$ denote the full Fock space on $\bC^d$. Let $L_1,\ldots,L_d\in B(F^2_d)$ denote the standard left creation operators, which satisfy $\sum_{k=1}^d L_k L_k^*\leq I$. The closed unital subalgebra $\A_d\subset B(F^2_d)$ generated by $L_1,\ldots,L_d$ is Popescu's non-commutative disc algebra \cite{popescu1989},\cite{popescu1996}. It is easily verified that the set $\{ab^*:a,b\in \A_d\}$ spans a dense subset of $\rC^*(\A_d)$.

The symmetric part of the full Fock space $F^2_d$ can be identified with the Drury--Arveson space $H^2_d$ of analytic functions on the open unit ball $\bB_d\subset \bC^d$ \cite{arveson1998} (see \cite{hartz2022} for a modern survey). If we denote by $\B_d\subset B(H^2_d)$ the norm-closure of the polynomial multipliers on this space, then there is a unital completely contractive homomorphism $\Gamma: \A_d\to \B_d$ sending each $L_k$ to the operator $Z_k$ of multiplication by the $k$-th coordinate function (see for instance \cite[Theorem 2.2]{CT2019cyclic} for a proof). In particular, we have
$
\sum_{k=1}^d Z_kZ_k^*\leq I .
$

By \cite[Theorem 5.7]{arveson1998}, we see that the set $\{ab^*:a,b\in \B_d\}$ spans a dense subset of $\rC^*(\B_d)$. Furthermore,  \cite[Theorem 5.7]{arveson1998} implies that the closed unit ball is contained in $W(Z_1,\ldots,Z_d)$, and
 hence in $W(L_1,\ldots,L_d)$ also. Conversely, the inequalities
 \[
 \sum_{k=1}^d L_kL_k^*\leq I \qand \sum_{k=1}^d Z_kZ_k^*\leq I 
 \]
imply that the joint numerical ranges are contained in $\ol{\bB_d}$. Therefore, 
\[
W(L_1,\ldots,L_d)=W(Z_1,\ldots,Z_d)=\ol{\bB_d}
\]
and the extreme points of this set are precisely those on the unit sphere. 

Finally, let $\L_d=\spn\{I,L_1,\ldots,L_d\}\subset \A_d$ and $\Z_d=\spn\{I,Z_1,\ldots,Z_d\}\subset \B_d$.  Applying Proposition \ref{P:WArv}, we conclude that $\L_d$ has the pure extension property in $\rC^*(\A_d)$, while $\Z_d$ has the pure extension property in $\rC^*(\B_d)$.
\qed
\end{example}

A natural question that we are unable to answer at present is whether the algebras $\A_d$ and $\B_d$ above  have the pure extension property inside their respective $\rC^*$-algebras.

\subsection{The case of $2\times 2$ matrices}

In this section, we investigate the pure extension property for subspaces of $\bM_2$. We give a somewhat streamlined version of the argument given in \cite{bal2022},\cite{krisko2022}.

We start with a reduction.

\begin{lemma}\label{L:2dim}
The following statements are equivalent.
\begin{enumerate}[{\rm (i)}]
\item Every $2$-dimensional unital subspace $M\subset \bM_2$ satisfying $\rC^*(M)=\bM_2$ has the pure extension extension property in $\rC^*(M)$.
\item Every unital subspace $M\subset \bM_2$  has the pure extension extension property in $\rC^*(M)$.
\end{enumerate}
\end{lemma}
\begin{proof}
Assume that (i) holds and let $M\subset \bM_2$ be a unital subspace. We must show that $M$ has the pure extension property in $\rC^*(M)$.

If $\dim M=1$, then $M=\bC I=\rC^*(M)$ and the desired conclusion is trivial. Likewise, if $\dim M=4$ then $M=\rC^*(M)$, and once again the conclusion is trivial. It thus only remains to deal with the case where $\dim M=3$.

Let $S=M+M^*\subset \bM_2$. It is well known that any state on $M$ has a unique extension to a state on $S$. We may therefore assume without loss of generality that $S$ is a proper subspace of $\bM_2$. In particular, $\dim S<4$ and we must have $\dim M=\dim S=3$, so $M=S$ and $M$ is in fact self-adjoint. 

Thus, there are self-adjoint elements $g,h\in M$ such that $M=\spn\{I,g,h\}$. Put $a=g+ih\in M$. Then, $M=N+N^*$, where $N=\spn\{I,a\}$. We claim that $N$ has the pure extension property in $\rC^*(N)=\rC^*(M)$. To see this, observe  that $N$ has dimension $2$. If $\rC^*(N)\cong\bC\oplus\bC$, then the claim holds trivially. The alternative is when $\rC^*(N)=\bM_2$, in which case the claim follows from our assumption. 

Finally, observe that a pure state on $M$ restricts to a pure state on $N$, so that $M$ also has the pure extension property in $\rC^*(M)$.
\end{proof}

We can now proceed to show that the pure extension property is always satisfied in $\bM_2$.

\begin{theorem}\label{T:M2}
Let $M\subset \bM_2$ be a unital subspace. Then, $M$ has the pure extension property in $\rC^*(M)$.
\end{theorem}
\begin{proof}
By virtue of Lemma \ref{L:2dim}, we may assume that $M$ has dimension $2$ and that $\rC^*(M)=\bM_2$. Choose an element $g\in M$ such that $M=\spn\{I,g\}$. Up to unitary equivalence, we may assume that $g$ is upper triangular, so that
\[
g=\begin{bmatrix}
g_{11} & g_{12} \\\ 0 & g_{22}
\end{bmatrix}
\] 
for some complex numbers $g_{11}, g_{12}, g_{22}$. Note that $g_{12}\neq 0$, for otherwise $\rC^*(M)$ is commutative, contrary to our assumption. 
Thus, replacing $g$ by 
\[
\frac{1}{g_{12}}g-\frac{g_{22}}{g_{12}}I
\]
we may assume that
\[
g=\begin{bmatrix}
te^{i\gamma} & 1 \\ 0 & 0
\end{bmatrix}
\]
for some $t\geq 0$ and $0\leq \gamma<2\pi$. 

To establish the desired result, we fix a pure state $\omega$ on $\bM_2$ whose restriction to $M$ is pure, and we aim to show that $\omega$ is uniquely determined by the number $\lambda=\omega(g)$.

Consider the numerical range of $g$, namely
\[
W=\{\phi(g):\phi\in \S(\bM_2)\}\subset \bC.
\]
The pure states on $\bM_2$ are precisely the vector states, so by the Krein--Milman theorem we see that $W$ coincides with the convex hull in $\bC$ of the numbers of the form $\langle g\xi,\xi\rangle$ where $\xi=(s,e^{i\alpha}\sqrt{1-s^2})\in \bC^2$ for some $0\leq s\leq 1$ and $ 0\leq \alpha<2\pi$. In other words, $W$ is the convex hull in $\bC$ of 
\[
\{te^{i\gamma}s^2+se^{i\alpha}\sqrt{1-s^2}: 0\leq s\leq 1, 0\leq \alpha<2\pi\}.
\]
Let $0\leq r\leq 1$ and $0\leq \theta<2\pi$ be chosen such that $\omega$ is the vector state corresponding to  $(r,e^{i\theta}\sqrt{1-r^2})\in \bC^2$. Our goal now becomes to show that $r$ and $\theta$ are uniquely determined by $\lambda$.

Since $\omega$ is pure on $M$, the number
\[ \lambda=\omega(g)=te^{i\gamma}r^2+re^{i\theta}\sqrt{1-r^2}\]
is an extreme point of $W$. On the other hand, it is known that $W$ is an ellipse (see for instance \cite{li1996}), so that $\lambda$ is an exposed point of $W$, in the sense that there is a supporting line for $W$ that intersects $W$ only at $\lambda$. In other words, there is a complex number $z$ such that 
\begin{equation}\label{Eq:exp}
\re (z\lambda)> \re(z\mu)
\end{equation}
for every $\mu\in W$ distinct from $\lambda$. 
Choose $\rho>0$ and $0\leq \beta<2\pi$ such that $z=\rho e^{i\beta}$.

Define now a function $f:[0,2\pi)\times [0,1]\to\bR$ as
\[
f(\alpha,s)=\re \left(z\left\langle g \begin{pmatrix} s\\ e^{i\alpha}\sqrt{1-s^2}\end{pmatrix}, \begin{pmatrix} s\\ e^{i\alpha}\sqrt{1-s^2}\end{pmatrix}\right\rangle\right)
\]
for each $\alpha\in [0,2\pi), s\in [0,1]$. By virtue of \eqref{Eq:exp}, to show that $r$ and $\theta$ are uniquely determined by $\lambda$, it suffices to show that $f$ attains its global maximum at a unique point $(\alpha_*,s_*)$.

We compute for $0\leq \alpha< 2\pi$ and $0\leq s\leq 1$ that
\begin{align*}
f(\alpha,s)&= \rho (s^2 t\cos(\gamma+\beta)+s\sqrt{1-s^2}\cos(\alpha+\beta)).
\end{align*}
Let now $\alpha_*\in [0,2\pi)$ be the unique number such that $\cos(\alpha_*+\beta)=1$. Then, for each $0\leq \alpha<2\pi$ distinct from $\alpha_*$ we see that
\[
f(\alpha,s)< f(\alpha_*,s), \quad 0\leq s\leq 1.
\]
Define $F:[0,1]\to\bR$ as
\[
F(s)=f(\alpha_*,s), \quad 0\leq s\leq 1.
\]
We need to show that $F$ attains its maximum at a unique point $s_*\in  [0,1]$. Differentiating, for $0<s<1$, we find
\[
F'(s)=\rho \left(2t s \cos(\gamma+\beta)+\frac{1-2s^2}{\sqrt{1-s^2}}\right)
\]
so that $F'(s)=0$  if and only if
\begin{equation}\label{Eq:crit}
2t s \cos(\gamma+\beta)=\frac{2s^2-1}{\sqrt{1-s^2}}.
\end{equation}
This equation implies
\[
Cs^4-Cs^2+1=0
\]
where $C=4(1+t^2 \cos^2(\gamma+\beta))$, so that
\[
s^2=\frac{1}{2}\pm\frac{1}{2}\sqrt{\frac{C-4}{C}}.
\]
On the other hand, since $0<s<1$, we note from \eqref{Eq:crit} that $s^2-1/2$ has the same sign as $\cos(\gamma+\beta)$. 

When $\cos(\gamma+\beta)\geq 0$, we see that 
\[
s_+=\sqrt{\frac{1}{2}+\frac{1}{2}\sqrt{\frac{C-4}{C}}}
\]
is the only possible critical point of $F$ in $(0,1)$ and
\[
F(s_+)=\frac{1}{4}(\sqrt{C-4}+\sqrt{C})
\]
so
\[
F(0)=0< F(1)=\frac{1}{2}\sqrt{C-4}< F(s_+).
\]

When $\cos(\gamma+\beta)< 0$, we see that 
\[
s_-=\sqrt{\frac{1}{2}-\frac{1}{2}\sqrt{\frac{C-4}{C}}}
\]
is the only possible critical point of $F$ in $(0,1)$ and 
\[
F(s_-)=\frac{1}{4}(-\sqrt{C-4}+\sqrt{C})
\]
so
\[
F(1)=-\frac{1}{2}\sqrt{C-4}<F(0)=0<F(s_-).
\]
In either case, we see that the maximum value of $F$ is attained at a unique point in $[0,1]$, as desired.
\end{proof}

In summary, any unital subspace of $\bM_2$ has the pure extension property, while there are subspaces of $\bM_4$ without this property \cite{sherman2023}. We do not know what happens for subspaces in $\bM_3$.

\section{Excisions}\label{S:excision}

The previous section illustrates that in order to obtain a complete answer to our main question, one must look further than the pure extension property. In this section, we draw inspiration from another source, based on the following notion.

Let $B$ be a unital $\rC^*$-algebra and let $\omega$ be a state on $B$. A contractive net $(e_\lambda)$ in $B$ is an \emph{excision} for $\omega$ if $\lim_\lambda\omega(e_\lambda)=1$ and
\[
\lim_\lambda \|e^*_\lambda (b-\omega(b)I)e_\lambda\|=0, \quad b\in B.
\]
Equivalently, the second condition means that
\[
\limsup_{\lambda} \|e_\lambda^* b e_\lambda\|\leq |\omega(b)|, \quad b\in B.
\]
Replacing nets with sequences, this notion was studied in \cite{clouatre2018lochyp} under the name ``characteristic sequence". The argument used to prove \cite[Theorem 4.2]{clouatre2018lochyp} adapts verbatim to show that the state $\omega$ is automatically pure if it admits an excision; this fact will be used below. 

We record the following basic well-known facts.

\begin{lemma}\label{L:suppproj}
Let $B$ be a unital $\rC^*$-algebra and let $\omega$ be a pure state on $B$.  If $\xi\in B^{**}$ is a contraction, then $\omega(\xi)=1$ if and only if $\xi \fl_\omega=\fl_\omega$.
\end{lemma}
\begin{proof}
Assume $\xi \in B^{**}$ is a contraction such that $\omega(\xi)=1$.  Then, $\xi$ lies in the multiplicative domain of $\omega$ \cite[Theorem 3.18]{paulsen2002}, so $\omega((I-\xi)^*(I-\xi))=0$ and  $(I-\xi)\fl_\omega=0$, or $\xi \fl_\omega=\fl_\omega$. Conversely, if $\xi\fl_\omega=\fl_\omega$, then $\omega((I-\xi)^*(I-\xi))=0$ which implies $\omega(\xi)=1$.
\end{proof}

\begin{lemma}\label{L:limitexcision}
Let $B$ be a unital $\rC^*$-algebra and let $\omega$ be a pure state on $B$. Let $(e_\lambda)$ be an excision for $\omega$. Then, $(e_\lambda)$ converges to $\fl_\omega$ in the weak-$*$ topology of $B^{**}$.
\end{lemma}
\begin{proof}
By compactness of the closed unit ball of $B^{**}$ in the weak-$*$ topology, it suffices to show that $\fl_\omega$ is the only cluster point of the net $(e_\lambda)$. Fix such a cluster point $\xi\in B^{**}$, so that $\|\xi\|\leq 1$. Upon passing to cofinal subnet if necessary, we may thus assume that $(e_\lambda)$ converges to $\xi$ in the weak-$*$ topology of $B^{**}$.
Because
\[
\omega(\xi)=\lim_\lambda\omega(e_\lambda)=1
\]
we have $\omega(\xi^*)=1$ and we conclude  from Lemma \ref{L:suppproj}  that $\xi^* \fl_\omega=\fl_\omega$, or $\fl_\omega=\fl_\omega \xi$. It thus only remains to show that $(I-\fl_\omega)\xi=0$. 

Since $\omega$ is assumed to be pure, the projection $\fl_\omega$ is closed \cite[Lemma 2.6]{CTh2023}. We may then choose an increasing net $(b_\mu)$ of positive contractions in $B$ converging to $I-\fl_\omega$ in the weak-$*$ topology of $B^{**}$.  Let $\phi$ be a state on $B$. We find
\begin{align*}
|\phi((I-\fl_\omega)\xi)|^2&=\lim_\mu\lim_\lambda |\phi(b_\mu e_\lambda)|^2\leq\limsup_\mu \limsup_\lambda \phi(e_\lambda^*b_\mu^2 e_\lambda)\\
&\leq \limsup_\mu \omega(b_\mu^2)\leq \limsup_\mu\omega(b_\mu)\\
&=\omega(I-\fl_\omega)=0
\end{align*}
by virtue of the Schwarz inequality and the excising property of $(e_\lambda)$. Since $\phi$ was an arbitrary state on $B$, this implies that $(I-\fl_\omega)\xi=0$ as desired.
\end{proof}

We also require the following.

\begin{theorem}[Anderson]\label{T:andersonexc}
Let $B$ be a unital $\rC^*$-algebra and let $A\subset B$ be a unital $\rC^*$-subalgebra. Let $\omega$ be a pure state on $B$. Then, the following statements are equivalent.
\begin{enumerate}[{\rm (i)}]
\item $\omega|_A$ admits a unique extension to a state on $B$.
\item $\omega$ admits an excision in $A$.
\end{enumerate}
\end{theorem}
\begin{proof}
If either (i) or (ii) holds, then $\omega|_A$ is pure. Hence, the result follows from  \cite[Theorem 3.2]{anderson1979}.
\end{proof}

This result makes it clear that excisions are  meaningful objects for the purposes of answering our main question on unique state extensions, and so we examine them a bit closer. In particular, we aim to relate their existence to a form of noncommutative peaking, which we describe next.

Let $B$ be a unital $\rC^*$-algebra and let $\S(B)$ denote its state space.  Let $\omega$ be a state on $B$, and let $\fl_\omega\in B^{**}$ denote its left support projection, that is the unique projection in $B^{**}$ satisfying 
\[
B^{**}(I-\fl_\omega)=\{x\in B^{**}:\omega(x^*x)=0\}.
\]
It follows that $\omega(\fl_\omega)=1$ and that $\fl_\omega$ lies in the multiplicative domain of $\omega$.

Let $X\subset B$ be any subset. 
Inspired by \cite{hay2007}, we say that $\fl_\omega$ is a \emph{$X$-peak projection} when there is $x\in X$ with $\|x\|=1$ such that $x\fl_\omega=\fl_\omega$ and $\|xp\|<1$ for every closed projection $p\in B^{**}$ orthogonal to $\fl_\omega$.  This property may be reformulated entirely in terms of states, as follows.

\begin{proposition}\label{P:peakprojnonorthog}
Assume that $\omega$ is a pure state on $B$. Then, the  following statements are equivalent.
\begin{enumerate}[{\rm (i)}]
\item $\fl_\omega$ is an $X$-peak projection.
\item There is $x\in X$ with $\|x\|=1$ such that $\omega(x)=1$ and $\phi(x^*x)<1$ if $\phi$ is a state on $B$ satisfying $\phi(\fl_\omega)=0$.
\item There is $x\in X$ such that $\omega(x)=1$ and $\phi(x^*x)<1$ if $\phi$ is a state on $B$ distinct from $\omega$.
\end{enumerate}
If we assume in addition that $X$ is a unital subspace, then these statements are further equivalent to the following.
\begin{enumerate}
\item[{\rm (iv)}]   There is $x\in X$ with $\|x\|=1$ such that $\omega(x)=1$ and $|\phi(x)|<1$ if $\phi$ is a state on $B$ satisfying $\phi(\fl_\omega)=0$.
\item[{\rm (v)}]   There is $x\in X$ with $\|x\|=1$ such that $\omega(x)=1$ and $|\phi(x)|<1$ if $\phi$ is a state on $B$ distinct from $\omega$.
\end{enumerate}
\end{proposition}
\begin{proof}
(i) $\Leftrightarrow$ (ii): This equivalence is contained in \cite[Theorem 5.1]{hay2007}.

(iii) $\Rightarrow$ (ii): We need only verify that $\|x\|=1$, which in this setting is equivalent to $\omega(x^*x)\leq 1$.  This is trivially true if $x$ is a scalar multiple of the identity, so we may suppose that $B$ has dimension bigger than $1$, in which case there exists a state $\psi$  on $B$ distinct from $\omega$. Given $0<\eps<1$, we see that $(1-\eps) \omega+\eps\psi$ is a state distinct from $\omega$, so by assumption 
\[
(1-\eps)\omega(x^*x)+\eps\psi(x^*x)<1.
\]
Letting $\eps\to 0$, we find $\omega(x^*x)\leq 1$ as desired.

(ii) $\Rightarrow$ (iii): Let $\phi$ be a state on $B$ distinct from $\omega$. By \cite[Lemma 2.2]{clouatre2018lochyp}, we see  that $\phi(\fl_\omega)<1$. Define $\psi:B\to \bC$ as \[\psi=\frac{1}{\phi(I-\fl_\omega)}\phi((I-\fl_\omega)\cdot (I-\fl_\omega)).\] Clearly, $\psi$ is a state on $B$ with $\psi(\fl_\omega)=0$, so by assumption $\psi(x^*x)<1$.
Next, by Lemma \ref{L:suppproj} we have that $x\fl_\omega=\fl_\omega$, which forces $x$ to commute with $\fl_\omega$ since $\|x\|=1$. We conclude that $x^*x=\fl_\omega+(I-\fl_\omega)x^*x(I-\fl_\omega)$, so that
\begin{align*}
\phi(x^*x)=\phi(\fl_\omega)+\phi(I-\fl_\omega)\psi(x^*x)<\phi(\fl_\omega)+\phi(I-\fl_\omega)=1.
\end{align*}

(ii) $\Rightarrow$ (v): The Schwarz inequality implies that $|\phi(x)|^2\leq \phi(x^*x)$ for any state $\phi$ on $B$.

(v) $\Rightarrow$ (iv): This is trivial.

\noindent Assume henceforth that $X$ is a unital subspace. 

(iv) $\Rightarrow$ (ii):  There is $x\in X$ with $\|x\|=1$ such that $\omega(x)=1$ and $|\phi(x)|<1$ if $\phi$ is a state on $B$ satisfying $\phi(\fl_\omega)=0$. As in \cite[Lemma 3.7]{hay2007}, we consider $y=(I+x)/2\in X$. Then, $\|y\|\leq 1$ and $\omega(y)=1$. We compute
\[
y^*y=\frac{1}{4}(I+x+x^*+x^*x)
\]
so that, given a state $\phi$ on $B$, we have
\[
\phi(y^*y)=\frac{1}{4}(1+2\re \phi(x)+\phi(x^*x))\leq \frac{1}{4}(2+2|\phi(x)|)\leq 1.
\]
Hence, $\phi(y^*y)=1$ implies that $|\phi(x)|=1$, and in turn we find $\phi(\fl_\omega)\neq 0$.
\end{proof}

One consequence of the previous result is that the notion of peak projection considered in \cite{clouatre2018lochyp} coincides with that introduced earlier in \cite{hay2007}, and hence with our current working definition, contrary to the guess formulated by the author at the time of writing \cite{clouatre2018lochyp}.

We now arrive to our result connecting excision and peaking. In the proof, we use a highly non-trivial fact about the existence of certain approximate units. It is a combination, due to Blecher, Hay, Neal and Read (and various subsets thereof), of known deep results from \cite{hay2007},\cite{BHN2008} and \cite{read2011}.

\begin{proposition}\label{P:excisionpeak}
Let $B$ be a unital $\rC^*$-algebra and let $A\subset B$ be a unital norm-closed subalgebra. Let $\omega$ be a pure state on $B$. Then, the following statements are equivalent.
\begin{enumerate}[{\rm (i)}]
\item There is a net  in $A$ forming an excision for $\omega$.

\item The projection $\fl_\omega$ lies in $A^{\perp\perp}$.
\end{enumerate}
\end{proposition}
\begin{proof}
(i) $\Rightarrow$ (ii): This follows at once from Lemma \ref{L:limitexcision}, since $A^{\perp\perp}$ coincides with the weak-$*$ closure of $A$ inside of $B^{**}$.

(ii) $\Rightarrow$ (i): The projection $\fl_\omega$ is  closed since $\omega$ is pure \cite[Lemma 2.6]{CTh2023}. By \cite[Corollary 2.25]{BR2011}, there is a contractive approximate right unit $(I-e_\lambda)$ for the left ideal $L_\omega=\{b\in B:b\fl_\omega=0\}$ converging to $I-\fl_\omega$ in the weak-$*$ topology of $B^{**}$ and with the property that each $e_\lambda$ is contractive. Note then that $\lim_\lambda\omega(e_\lambda)=\omega(\fl_\omega)=1$ and
\[
\lim_\lambda \|be_\lambda\|=0, \quad b\in L_\omega.
\]
We can now proceed as in the proof of  \cite[Theorem 2.3]{AAP1986}. Let $b\in B$. Then, $b-\omega(b)I\in \ker \omega$. The purity of $\omega$ implies that $\ker \omega=L_\omega^*+L_\omega$ \cite[Proposition 2.9.1]{dixmier1977}, so we may find $s,t\in L_\omega$ with $b-\omega(b)I=s^*+t.$ We obtain
\begin{align*}
\limsup_\lambda\| e_\lambda^*(b-\omega(b)I)e_\lambda\|&= \limsup_\lambda \|e_\lambda^* (s^*+t) e_\lambda\|\\
&\leq \limsup_\lambda (\|e_\lambda^* s^*\|+\|t e_\lambda\|)\\
&= \limsup_\lambda (\|se_\lambda\|+\|t e_\lambda\|)\\
&=0.
\end{align*}
We conclude that $(e_\lambda)$ is an excision for $\omega$.
\end{proof}

The following consequence summarizes the relationships between our  notions of interest.

\begin{corollary}\label{C:excisionequiv}
Let $B$ be a unital $\rC^*$-algebra and let $A\subset B$ be a unital, separable, norm-closed subalgebra. Let $\omega$ be a pure state on $B$. Consider the following statements.
\begin{enumerate}[{\rm (i)}]
\item $\omega$ admits an excision in $A$.

\item $\fl_\omega$ is an $A$-peak projection.

\item The restriction  $\omega|_A$ admits a unique extension to a state on $B$.

\end{enumerate}
Then, we have that 
\[
{\rm (i)} \Leftrightarrow{\rm (ii)}\Rightarrow {\rm (iii)}.
\]
\end{corollary}
\begin{proof}
(i) $\Leftrightarrow$ (ii): This is Proposition \ref{P:excisionpeak} combined with \cite[Corollary 3.3]{CTh2022}.

(ii) $\Rightarrow$ (iii): This can be found in \cite[Proposition 3.2]{CTh2023}.

\end{proof}

When $A$ is a $\rC^*$-algebra, then statements (i),(ii) and (iii) are equivalent by virtue of Theorem \ref{T:andersonexc}. We  show next that the implications (iii)$\Rightarrow$(ii) fails in the nonselfadjoint setting.

\begin{example}\label{E:ueppeakproj}
Let $A\subset \bM_2$ be the algebra of upper triangular matrices. Then, $A+A^*=\bM_2$, which readily implies that every state on $A$ has a unique extension to a state on $\bM_2$. Denote by $\{e_1,e_2\}$ the standard orthonormal basis of $\bC^2$, and put $\xi=\frac{1}{\sqrt{2}}(e_1+e_2)$. Let $\omega:\bM_2\to \bC$ be the pure state defined as
\[
\omega(t)=\langle t\xi,\xi \rangle,\quad t\in \bM_2.
\]
Then, $\fl_\omega$ is the rank-one projection onto $\bC \xi$, which does not lie in $A=A^{\perp\perp}$. We conclude from Proposition \ref{P:excisionpeak} that $\omega$ does not admit an excision in $A$, so in particular $\fl_\omega$ is not an $A$-peak projection by Corollary \ref{C:excisionequiv}. \qed
\end{example}

In light of this example, we see that excisions and peak projections do not characterize unique extensions for states, and hence do not provide an answer to our main question. Indeed, a weaker notion of peaking is needed.

\section{Weaker notions of noncommutative peak points}\label{S:ncpeak}
In the previous section, given a state $\omega$ on a unital $\rC^*$-algebra $B$ along with a unital subspace $M\subset B$,  we saw that the property that $\fl_\omega$ be an $M$-peak projection is typically strictly stronger than that of $\omega|_M$ admitting a unique extension to a state on $B$. We are thus lead to relax this notion of peaking, in hopes of precisely characterizing the unique extension property. Two natural such relaxations are examined in this section.

As before, we denote by $\S(B)$ the state space of $B$. Let $X\subset B$ and  $E\subset \S(B)$ be  non-empty subsets. We say that $E$ is an \emph{$X$-peak set} if there is $x\in X$ such that $\omega(x)=1$ for every $\omega\in E$, while $|\psi(x)|<1$ for every $\psi\in \S(B)\setminus E$. We record a basic property of these sets.

\begin{lemma}\label{L:peakpure}
Let $B$ be a unital $\rC^*$-algebra and let $X\subset B$ be a non-empty subset. Then, any $X$-peak set must contain a pure state on $B$.
\end{lemma}
\begin{proof}
Let $E\subset \S(B)$ be a non-empty $X$-peak set. Then, there is $x\in X$ such that $\omega(x)=1$ for every $\omega\in E$, while $|\psi(x)|<1$ for every $\psi\in \S(B)\setminus E$. This means that the weak-$*$ continuous convex function 
\[
\phi\mapsto |\phi(x)|, \quad \phi\in \S(B)
\]
must attain its maximum on $E$. By virtue of Bauer's maximum principle \cite[Theorem I.5.3]{alfsen1971}, we conclude that $E$ contains a pure state.
\end{proof}

Next, let $M\subset B$ be a unital subspace. Put $\Q(M)=\{a^*a:a\in M\}$. Let also $\omega$ be a state on $B$. Although the above definition of a peak set allows for arbitrary choices of $X$ and $E$, for the rest of the paper we will be solely interested in  the cases where $X$ is either $M$ or $\Q(M)$, and $E$ is either $\{\omega\}$ or a subset of the collection $E_\omega$ of states on $B$ extending $\omega|_M$. 
When $E=\{\omega\}$ is an $X$-peak set, then we say that  $\omega$ is an \emph{$X$-peak state}.

\begin{corollary}\label{C:projstate}
Let $B$ be a unital $\rC^*$-algebra and let $M\subset B$ be a unital subspace. Let $\omega$ be a  state on $B$ such that $\fl_\omega$ is an $M$-peak projection. Then, $\omega$ is both an $M$-peak state and a $\Q(M)$-peak state.
\end{corollary}
\begin{proof}
This is an immediate consequence of the Schwarz inequality along with Proposition \ref{P:peakprojnonorthog}. 
\end{proof}
We thus have the following diagram.
\begin{equation}\label{Eq:imp1}
\begin{tikzcd}[arrows=Rightarrow]
\substack{\fl_\omega \text{  is}\\ M\text{-peak projection}} &  \substack{\omega\text{  has}\\ M\text{-peak support}} \arrow[l] \arrow[r] & \substack{\omega \text{  is}\\ \Q(M)\text{-peak state}}
\end{tikzcd}
\end{equation}

In Section \ref{S:peakuep}, we shall see that neither implication can be reversed in general (see \eqref{Eq:diagram}).

For now, we illustrate how certain special peak sets arise. Note that if $x\in M$ has norm $1$, then the set $\{\phi\in \S(B):\phi(x^*x)=1\}$ is clearly a $\Q(M)$-peak set. Something similar is true for $M$-peak sets, as we show next.

Given an element $b\in B$, we define its \emph{numerical radius} as
\[
\nr(b)=\sup_{\phi\in \S(B)}|\phi(b)|,
\]
which then satisfies
\[
\|b\|/2\leq \nr(b) \leq \|b\|. 
\]

\begin{lemma}\label{L:halftrick}
Let $B$ be a unital $\rC^*$-algebra and let $M\subset B$ be a unital subspace. Let $x\in M$ be an element such that $\nr(x)=1$. Then, the set
$
\{\phi\in \S(B):\phi(x)=1\}
$
is an $M$-peak set.
\end{lemma}
\begin{proof}
Let $y=\frac{1}{2}(I+x)\in M$, so that $\nr(y)\leq 1$. An elementary calculation reveals that a state $\phi$ on $B$  satisfies $|\phi(y)|=1$ precisely when $\phi(x)=\phi(y)=1$. Hence,  $\{\phi\in \S(B):\phi(x)=1\}$ is an $M$-peak set.
\end{proof}

We now define four sets for an element $b\in B$:
\[
N^\Q_b=\{\phi\in \S(B):\phi(b^*b)=\|b\|^2\}, \quad R^\Q_b=\{\phi|_M:\phi\in N^\Q_b\}
\]
and
\[
N^{\nr}_b=\{\phi\in \S(B):|\phi(b)|=\nr(b)\}, \quad R^{\nr}_b=\{\phi|_M:\phi\in N^{\nr}_b\}.
\]
Given a state $\omega$ on $B$, we define $\Sigma^\Q_\omega\subset B$ to be the subset  of elements $b$ with $\|b\|=1$ for which $R^\Q_b=\{\omega|_M\}$. Likewise, $\Sigma^{\nr}_\omega\subset B$ is the subset  of elements $b$ with $\nr(b)=1$ for which $R^{\nr}_b=\{\omega|_M\}$. The following will be of use in Section \ref{S:abundance}.

\begin{proposition}\label{P:peakset}
Let $B$ be a unital $\rC^*$-algebra and let $M\subset B$ be a unital subspace.
Let $\omega$ be a state on $B$. Consider the following statements.
\begin{enumerate}[{\rm (i)}]
\item $\omega$ is a $\Q(M)$-peak state.
\item There is a non-empty $\Q(M)$-peak set contained in $E_\omega$.
\item There is $x\in M\cap \Sigma^\Q_\omega$ such that $\{\omega|_M\}= R^\Q_x$.
\item $\omega$ is an $M$-peak state.
\item There is a non-empty $M$-peak set contained in $E_\omega$.
\item There is $x\in M\cap \Sigma^{\nr}_\omega$ such that $\{\omega|_M\}= R^{\nr}_x$.

\end{enumerate}
Then, we have that ${\rm (i)} \Rightarrow{\rm(ii)} \Leftrightarrow{\rm(iii)}$ and ${\rm (iv)} \Rightarrow{\rm(v)} \Leftrightarrow{\rm(vi)}$. 
When $\omega|_M$ admits a unique extension to a state on $B$, then ${\rm (i)} \Leftrightarrow{\rm(ii)} \Leftrightarrow{\rm(iii)}$ and ${\rm (iv)} \Leftrightarrow{\rm(v)} \Leftrightarrow{\rm(vi)}$. 
\end{proposition}

\begin{proof}
(i)$\Rightarrow$(ii) and (iv)$\Rightarrow$(v): The peak set can trivially be taken to be $\{\omega\}$ in both cases.

(ii)$\Rightarrow$(iii): By assumption, there is $x\in M$ such that $\phi(x^*x)=1>\psi(x^*x)$ for every $\phi\in E$ and $\psi\in \S(B)\setminus E$. In particular, $\|x\|=1$ and we see that  $E=N^\Q_x$. Since $E\subset E_\omega$ this implies  $R^\Q_x=\{\omega|_M\}$ and $x\in \Sigma^\Q_\omega$.

(iii)$\Rightarrow$(ii): By assumption, there is $x\in \Sigma^\Q_\omega$ such that $\{\omega|_M\}= R^\Q_x$. Then,  $N^\Q_x\subset E_\omega$ is the desired non-empty $\Q(M)$-peak set.

(v)$\Rightarrow$(vi): By assumption, there is $x\in M$ such that $\phi(x)=1>|\psi(x)|$ for every $\phi\in E$ and $\psi\in \S(B)\setminus E$. In particular, $\nr(x)=1$ and  we see that  $E=N^{\nr}_x$. Since $E\subset E_\omega$ this implies  $R^{\nr}_x=\{\omega|_M\}$ and $x\in \Sigma^{\nr}_\omega$.

(vi)$\Rightarrow$(v): By assumption, there is $x\in \Sigma^{\nr}_\omega$ such that $\{\omega|_M\}= R^{\nr}_x$. Choose $\zeta\in \bC$ with $|\zeta|=1$ such that $\omega(\zeta x)=1$, and put $y=\zeta x\in M$. Then, $R^{\nr}_y=R^{\nr}_x=\{\omega|_M\}$. Applying Lemma \ref{L:halftrick}, we see that $\{\phi\in \S(B):\phi(y)=1\}$ is an $M$-peak set contained in $ E_\omega$, and it is non-empty since it contains $\omega$.

The final sentence is clear, since in this case we have $E_\omega=\{\omega\}$.
\end{proof}

\subsection{A concrete example on the Dirichlet space}

Our next goal is to analyze a concrete operator algebra of analytic functions through the lens of noncommutative peak point theory. This is predicated  on the following  basic observation.

\begin{proposition}\label{P:compactsQpeak}
Let $H$ be a Hilbert space and let $B\subset B(H)$ be a unital $\rC^*$-subalgebra containing the ideal $\fK$ of compact operators. Let $q:B\to B/\fK$ denote the quotient map. Let $\xi\in H$ be a unit vector, and let $\omega=\langle \cdot \xi,\xi\rangle$ be the corresponding pure vector state on $B$. For $b\in B$ with $\|b\|=1$, the following statements are equivalent.
\begin{enumerate}[{\rm(i)}]
\item $\omega(b^*b)=1>\psi(b^*b)$ for every state $\psi$ on $B$ distinct from $\omega$.
\item $\|q(b)\|<1$ and $\ker (b^*b-I)=\bC \xi$.
\end{enumerate}
\end{proposition}
\begin{proof}
(i)$\Rightarrow$(ii): If $\|q(b)\|=1$, then there is a state $\phi$ on $B/\fK$ such that $(\phi\circ q)(b^*b)=1$. The assumption then forces $\omega=\phi\circ q$, which is absurd as $\omega$ does not annihilate the rank-one projection onto $\bC \xi$. 
Next, we see that
$
\langle b^*b\xi,\xi\rangle=\omega(b^*b)=1
$
so by the Cauchy--Schwarz inequality we find $\bC \xi\subset \ker (b^*b-I)$. Given an arbitrary unit vector $\eta\in \ker(b^*b-I)$, the corresponding pure vector state $\psi=\langle \cdot \eta,\eta\rangle$ on $B$ clearly satisfies $\psi(b^*b)=1$, so that $\psi=\omega$ by assumption, and in turn $\eta\in \bC \xi$.

(ii)$\Rightarrow$(i): By \cite[Lemma 2.1]{CTh2023}, it suffices to fix a \emph{pure} state $\psi$ on $B$ distinct from $\omega$ and show that $\psi(b^*b)<1$. By basic representation theory of $\rC^*$-algebras, either $\psi=\phi\circ q$ for some state $\phi$ on $B/\fK$, or $\psi$ is a vector state. In the first case, we find  $\psi(b^*b)\leq \|q(b)\|^2<1$. Alternatively, if there is a unit vector $\eta\in H$ such that $\psi=\langle \cdot \eta,\eta\rangle$, then using $\psi\neq \omega$ we see that $\eta\notin \bC\xi=\ker(b^*b-I)$, so $\psi(b^*b)<1$.
\end{proof}

\begin{example}\label{E:Dirichlet}
Let $\D$ denote the Dirichlet space on the unit disc $\bD\subset\bC$. This is a classical object of study in complex function theory and harmonic analysis; the reader may consult \cite{EKMR2014} for more detail. We only recall here the salient facts the we require for our analysis. 

The space $\D$ is a reproducing kernel Hilbert space of holomorphic functions on $\bD$, and it contains the constant functions.
Every polynomial $p$ induces a bounded multiplication operator $M_p:\D\to \D$ with the property that $\|M_p\|\geq \sup_{z\in \bD}|p(z)|.$ There are some polynomials for which this inequality is strict (contrary to the case of the Hardy space on $\bD$). 

Let $\rA(\D)\subset B(\D)$ denote the norm-closed unital subalgebra generated by the polynomial multipliers, and let $\fT(\D)=\rC^*(\rA(\D))\subset B(\D)$. The $\rC^*$-algebra $\fT(\D)$ contains the ideal $\fK$ of compact operators on $\D$, and there is a unital surjective $*$-homomorphism $\pi:\fT(\D)\to \rC(\bT)$ with $\ker \pi=\fK$ such that $\pi(M_p)=p$ for every polynomial $p$ (\cite[Example 1 and Theorem 4.6]{GHX04}). Here, $\bT\subset \bC$ denotes the unit circle. We now examine the various notions of peaking introduced above for the $\rC^*$-algebra $\fT(D)$ relative to $\rA(\D)$.

Let $\omega$ be a pure state on $\fT(\D)$. As above, there are two possibilities for $\omega$.

If $\omega$ annihilates $\fK$, then there is a state $\tau$ on $\rC(\bT)$, necessarily pure since $\omega$ is, such that $\omega=\tau\circ \pi$. Then, $\tau$ is given as evaluation at a point in $\bT$. Combining \cite[Proposition 9.2]{DH2023} with \cite[Corollary 5.8]{BC2024}, we see that $\fl_\omega$ is an $\rA(\D)$-peak projection.

The remaining case is  that where there is a unit vector $\xi\in \D$ such that $\omega=\langle \cdot \xi,\xi\rangle$.   Note that if $a\in \rA(\D)$ satisfies $\omega(a)=1=\|a\|$, then by the Cauchy--Schwarz inequality we find $a\xi=\xi$. Since $\D$ consists of holomorphic functions on $\bD$, this forces $a=I$. In particular, we conclude from Proposition \ref{P:peakprojnonorthog} that $\fl_\omega$ is not an $\rA(\D)$-peak projection.   Moreover, $\omega|_{\rA(\D)}$ typically admits several extensions to a state on $\fT(\D)$. For instance, $\langle \cdot 1,1\rangle$ agrees with $\tau_0\circ \pi$ on $\rA(\D)$, where $\tau_0$ is the state of integration against arclength measure on $\bT$. In particular, this shows that $\omega$ is typically not an $\rA(\D)$-state.  Finally, by Proposition \ref{P:compactsQpeak}, we see that $\omega$ is a $\Q(\rA(\D))$-peak state if and only if there is $a\in \rA(\D)$ with $\|a\|=1>\|q(a)\|$ and $\ker(a^*a-I)=\bC \xi$. There are many such elements in $\rA(\D)$, as seen in \cite[Theorem 10.2]{AHMR2023}.\qed
\end{example}

Much of the analysis above can be replicated for a large class of spaces, namely the unitarily invariant complete Pick reproducing kernel Hilbert spaces on the open unit ball \cite{DH2023}. One crucial feature  of $\D$ is that \cite[Theorem 10.2]{AHMR2023} is available, which was required by Proposition \ref{P:compactsQpeak}. The validity of the corresponding statement in the setting of, say, the ubiquitous Drury--Arveson space, was raised as a question in \cite{CD2016duality}, and still does not seem to be known.

\section{Abundance of peak sets}\label{S:abundance}

We start this discussion with a simple observation, already noted in \cite{hay2007}.
Let  $A\subset \bM_2$ be the unital subalgebra of upper-triangular matrices with constant diagonal. Then, it is readily verified that there exist simply no $A$-peak projections in this case. If one is to meaningfully consider peak projections as the noncommutative analogues of peak points, then this dearth is a rather serious impediment, as it conflicts with the classical case of continuous functions. Indeed, peak points are always dense in the set of pure states with pure restrictions in the commutative case \cite[Corollary 8.4]{phelps2001}.  

In this section, we establish a noncommutative analogue of this density fact using the relaxed notions of peaking introduced in Section \ref{S:ncpeak}. To do so, we will rely on an important Banach space geometry result, which we now describe. Let $X$ be a normed space, and denote by $U\subset X$ the unit sphere. Recall that a point $x\in U$ is said to be \emph{smooth} if there is a unique bounded linear functional $\lambda:X\to\bC$ such that $\|\lambda\|=1=\lambda(x)$. A proof of the following may be found in \cite[Theorem 8.3]{phelps2001}, for instance.

\begin{theorem}[Mazur]\label{T:Mazur}
If $X$ is a  separable Banach space, then the smooth points of its unit sphere $U$ form a dense subset of type $G_\delta$ in $U$.
\end{theorem}

Let $B$ be a unital  $\rC^*$-algebra and let $M\subset B$ be a unital  subspace. 
Our first order of business will be to show that, by applying this result to an appropriate Banach space, we can establish that certain special  $M$-peak sets (of the type appearing in Proposition \ref{P:peakset}) are rather abundant. For this purpose, we recall that the numerical radius $\nr(\cdot)$ defines a norm on $M$, equivalent to the original norm coming from the $\rC^*$-algebra structure of $B$. 

\begin{lemma}\label{L:smoothstate}
Let $B$ be a unital  $\rC^*$-algebra and let $M\subset B$ be a unital subspace. Let $U$ denote the unit sphere of $(M,\nr(\cdot))$. Let $x$ be a smooth point of $U$. Then, the following statements hold.
\begin{enumerate}[{\rm (i)}]
\item There is a pure state $\omega$ on $B$ such that $|\omega(x)|=1$ and $\omega|_M$ is pure.
\item   If $\psi$ is a state on $B$ with $|\psi(x)|=1$, then $\omega|_M=\psi|_M$. 
\item $E_\omega$ contains a non-empty $M$-peak set.
\end{enumerate}
\end{lemma}
\begin{proof}
(i) By weak-$*$ compactness, we see that the set $\{\phi\in \S(M):|\phi(x)|=\nr(x)\}$ is non-empty. Further, it is a face of $\S(M)$, so by the Krein--Milman theorem there is a pure state $\omega_0$ on $M$ with $|\omega(x)|=\nr(x)=1$. It thus suffices to choose $\omega$ to be any pure state extension of $\omega_0$.

(ii) Note that both $\omega|_M$ and $\psi|_M$ are norm-one linear functionals on $(M,\nr(\cdot))$. Choose $\alpha,\beta\in \bC$ with $|\alpha|=|\beta|=1$ such that $\alpha\omega(x)=1=\beta\psi(x)$. Since $x$ is smooth, we see that $\alpha \omega|_M=\beta \psi|_M$. Evaluating at  the unit $I\in M$, we conclude that $\alpha=\beta$, whence $\omega|_M=\psi|_M$.

(iii) By (i) and (ii), we have $R^{\nr}_x=\{\omega|_M\}$. Hence by Proposition \ref{P:peakset} we see that $E_\omega$ contains a non-empty $M$-peak set.
\end{proof}

We now extract a sort of dual statement to Theorem \ref{T:Mazur}, which shows that certain special $M$-peak sets are always abundant.

\begin{theorem}\label{T:P0dense}
Let $B$ be a unital  $\rC^*$-algebra and let $M\subset B$ be a unital, separable, norm-closed  subspace. Let $\P_0$  be the set of pure states $\omega$ on $B$ for which there is a smooth point $x$ of the unit sphere of $(M,\nr(\cdot))$ satisfying $|\omega(x)|=1$. Let $\P(M)$ denote the set of pure states $\omega$ on $B$ such that $E_\omega$ contains a non-empty $M$-peak set.  Then, the sets $\{\omega|_M:\omega\in \P_0\}$ and $ \{\omega|_M:\omega\in \P(M)\} $ are both weak-$*$ dense in the pure states on $M$.  
\end{theorem}
\begin{proof}
Note that $\P_0\subset \P(M)$  and $\omega|_M$ is pure for every $\omega\in \P_0$ by Lemma \ref{L:smoothstate}, so it suffices to show that the weak-$*$ closure of $\{\omega|_M:\omega\in \P_0\}$ contains the pure states on $M$.

Let $Y=(M,\nr(\cdot))$, which is a separable Banach space since $\nr(\cdot)$ is a norm on $M$ equivalent to the original norm coming from the $\rC^*$-algebra $B$. Denote by $K$ the closed unit ball of $Y^*$.

We first claim the convex hull of $\{\zeta\omega|_M:\zeta\in \bC, |\zeta|=1,\omega \in \P_0\}$ is weak-$*$ dense in $K$. For otherwise, there is $y\in U$ such that 
\[
\sup_{\omega\in \P_0}|\omega(y)|< 1.
\]
By virtue of Theorem \ref{T:Mazur}, we may assume that $y$ is a smooth point of $U$. In this case, Lemma \ref{L:smoothstate} implies that there is $\omega\in \P_0$ with $|\omega(y)|=1$. This is absurd, and the claim is established. In particular, by Milman's converse \cite[Proposition 1.5]{phelps2001}, we  know that the extreme points of $K$ are contained in the weak-$*$ closure of $\{\zeta\omega|_M:\zeta\in \bC, |\zeta|=1,\omega \in \P_0\}$. 

Next, let $\phi$ be a state on $M$, and assume that $\phi=\frac{1}{2}(\lambda_1+\lambda_2)$ for $\lambda_1,\lambda_2\in K$. Then, using that $|\lambda_i(I)|\leq \nr(I)=1$ and
\[
1=\phi(I)=\frac{1}{2}(\re \lambda_1(I)+\re\lambda_2(I))\leq 1
\]
we see that $\lambda_1(I)=\lambda_2(I)=1$, so $\lambda_1$ and $\lambda_2$ are in fact states on $M$. This shows that $\S(M)$ is a weak-$*$ closed face of $K$, so that the pure states on $M$ are extreme points of $K$. Hence, the pure states on $M$ are contained in the weak-$*$ closure of $\{\zeta\omega|_M:\zeta\in \bC, |\zeta|=1,\omega \in \P_0\}$, by the previous paragraph. Since states are unital maps, the pure states on $M$ must in fact be contained in the weak-$*$ closure of $\{\omega|_M:\omega\in \P_0\}$. 
\end{proof}

Under an additional assumption $M$, we can in fact replace $M$-peak sets with $M$-peak states. 

\begin{corollary}\label{C:PMdense}
Let $B$ be a unital  $\rC^*$-algebra and let $M\subset B$ be a unital , separable, norm-closed  subspace.  Assume that $M$ has the pure extension property in $B$. Then, the $M$-peak states are weak-$*$ dense in the pure states on $B$ that restrict to be pure on $M$.
\end{corollary}
\begin{proof}
Because of the pure extension property, we note that $\P(M)$ consists of $M$-peak states. Consider the restriction map $\rho:\S(B)\to \S(M)$, and let $E\subset \S(B)$ denote the set of pure states on $B$ that restrict to be pure on $M$. In other words, $E$ is the preimage under $\rho$ of the pure states on $M$. Since every pure state on $M$ admits a unique extension to a pure state on $B$ by assumption, the map $\rho$ implements a weak-$*$ homeomorphism between $E$ and the pure states on $M$. By Theorem \ref{T:P0dense}, we see that $\rho(\P(M))$ is weak-$*$ dense in $\rho(E)$, so $\P(M)$ is in fact weak-$*$ dense in $E$.
\end{proof}

Note that if $B$ above is commutative and $B=\rC^*(M)$, then $M$ necessarily has the pure extension property, and the previous result recovers the classical fact that the peak points are dense in the Choquet boundary of $M$ \cite[Corollary 8.4]{phelps2001}.

Our next objective is to obtain an analogous abundance result for $\Q(M)$-peak sets. For this purpose, we require an adaptation of Theorem \ref{T:Mazur}. The key technical step is the following;  our argument  is closely modelled on the proof of \cite[Proposition 8.3]{phelps2001}. We use the same notation as in Proposition \ref{P:peakset}.

\begin{lemma}\label{L:dense}
Let $B$ be a  unital separable $\rC^*$-algebra and let $M\subset B$ be a norm-closed subspace.  Fix $u\in M$ with $\|u\|=1$ and $\eps>0$.
Then, the set
\[
D=\{x\in M: \|x\|=1, \re (\phi(u^*x)-\psi(u^* x))<\eps \text{ for every } \phi,\psi\in N^\Q_x\}
\]
is open and dense in the unit sphere of $M$.
\end{lemma}
\begin{proof}
First, observe that since $B$ is assumed to be separable, the weak-$*$ topology on the state space $\S(B)$ is metrizable. This will be used implicitly throughout.

Denote by $U$ the unit sphere of $M$.
Let $(y_n)$ be a sequence in $U\setminus D$ which converges in norm to some $y\in U$. Thus, for each integer $n\geq 1$  there are states $\phi_n,\psi_n\in N^\Q_{y_n}$ such that 
\[
\re (\phi_n(u^*y_n)-\psi_n(u^*y_n))\geq \eps.
\]
Let $\phi,\psi\in \S(B)$ be weak-$*$ cluster points of the sequences $(\phi_n)$ and $(\psi_n)$ respectively. Since $(y_n)$ converges to $y$ in norm,  it is easily verified that $\phi,\psi\in N^\Q_y$ and
$
\re(\phi-\psi)(u^*y)\geq \eps.
$
We conclude that $y\in U\setminus D$, whence $D$ is open.

Assume next for the sake of contradiction that $D$ is not dense in $U$, so that there is $x\in U\setminus D$ along with $\delta>0$ such that $y\in U\setminus D$ whenever $y\in U$ satisfies $\|x-y\|<\delta$. We put $x_1=x$ and  choose $\phi_1,\psi_1\in N^\Q_{x_1}$ such that
\[
\re (\phi_1(u^*x_1)-\psi_1(u^* x_1))\geq \eps.
\]
Suppose that, given $n\geq 1$, we have constructed $x_n\in U\setminus D$ such that 
$
\|x_1-x_n\|<\left( 1-\frac{1}{2^n}\right)\delta
$
along with $\phi_n,\psi_n\in N^\Q_{x_n}$ such that
\[
\re(\phi_n(u^*x_n)-\psi_n(u^* x_n))\geq \eps.
\]
Since both $x_n$ and $u$ have norm $1$, we may choose $\alpha_n>0$ small enough so that $x_n+\alpha_n u\neq 0$,
\begin{equation}\label{Eq:alpha}
\alpha_n (1- \phi_n(u^* u))<\eps
\end{equation}
and
\begin{equation}\label{Eq:alphanorm}
\frac{\|(1-\|x_n+\alpha_n u\|)x_n+\alpha_n u\|}{\|x_n+\alpha_n u\|}<\frac{\delta}{2^{n+1}}.
\end{equation}
Define
\[
x_{n+1}=\frac{1}{\|x_n+\alpha_n u\|}(x_n+\alpha_n u)\in U.
\]
Using \eqref{Eq:alphanorm}, we obtain
\begin{equation}\label{Eq:Cauchy}
\|x_n-x_{n+1}\| =\frac{\|(1-\|x_n+\alpha_n u\|)x_n+\alpha_n u\|}{\|x_n+\alpha_n u\|}<\frac{\delta}{2^{n+1}}
\end{equation}
and thus
\begin{align*}
\|x_1-x_{n+1}\|&\leq \|x_1-x_n\|+\|x_n-x_{n+1}\|< \left( 1-\frac{1}{2^n}\right)\delta+\frac{\delta}{2^{n+1}}\\
&= \left( 1-\frac{1}{2^{n+1}}\right)\delta<\delta
\end{align*}
so we see that $x_{n+1}\in U\setminus D$ by choice of $\delta$. In particular, we may choose $\phi_{n+1},\psi_{n+1}\in N^\Q_{x_{n+1}}$ such that
\[
\re(\phi_{n+1}(u^*x_{n+1})-\psi_{n+1}(u^*x_{n+1}))\geq \eps.
\]

By induction, we obtain a sequence $(x_n)$ in $U\setminus D$, a sequence $(\alpha_n)$ of positive numbers, and sequences $(\phi_n),(\psi_n)$ of states on $B$ such that
\begin{equation}\label{Eq:xn}
x_{n+1}=\frac{1}{\|x_n+\alpha_n  u\|}(x_n+\alpha_n u),
\end{equation}
\begin{equation}\label{Eq:phipsi}
\phi_n(x_n^*x_n)=\psi_n(x_n^*x_n)=1
\end{equation}
and
\begin{equation}\label{Eq:Re}
\re(\phi_n(u^*x_n)-\psi_n(u^* x_n))\geq \eps
\end{equation}
for every $n\geq 1$.
Using \eqref{Eq:xn} and \eqref{Eq:phipsi}, we observe that
\begin{align*}
1&\geq \re \phi_n(x_{n+1}^*x_{n+1})\\
&=\frac{1}{\|x_n+\alpha_n u\|^2}(\phi_n(x_n^*x_n)+2\alpha_n\re \phi_n (u^* x_n)+\alpha_n^2 \phi_n(u^* u))\\
&=\frac{1}{\|x_n+\alpha_n u\|^2}(1+2\alpha_n\re \phi_n(u^* x_n)+\alpha_n^2 \phi_n(u^* u))
\end{align*}
and
\begin{align*}
1&= \re \psi_{n+1}(x_{n+1}^*x_{n+1})\\
&=\frac{1}{\|x_n+\alpha_n u\|^2}(\psi_{n+1}(x_n^*x_n)+2\alpha_n\re \psi_{n+1} (u^* x_n)+\alpha_n^2\psi_{n+1}(u^* u))\\
&\leq \frac{1}{\|x_n+\alpha_n u\|^2}(1+2\alpha_n\re \psi_{n+1} (u^* x_n)+\alpha_n^2 \psi_{n+1}(u^* u)).
\end{align*}
Combining these inequalities we infer
\begin{align*}
2\re \phi_n(u^* x_n)+\alpha_n \phi_n(u^* u) &\leq 2\re\psi_{n+1}  (u^* x_n)+\alpha_n \psi_{n+1}(u^* u)\\
&\leq  2\re\psi_{n+1}  (u^* x_n)+\alpha_n 
\end{align*}
and therefore
\begin{equation}\label{Eq:psi}
\re( \phi_n(u^* x_n)- \psi_{n+1}  (u^* x_n))\leq\frac{1}{2} \alpha_n (1- \phi_n(u^* u))<\eps/2
\end{equation}
by virtue of \eqref{Eq:alpha}. 

Note now that \eqref{Eq:Cauchy} implies that the sequence $(x_n)$ in Cauchy, so it converges in norm to some $z\in U$. Choose $N$ large enough so that $\|x_n-z\|<\eps/16$ for each $n\geq N$. Then, \eqref{Eq:Re} and \eqref{Eq:psi} easily imply that
\[
\re\phi_n(u^*z)\geq \re \psi_n(u^* z)+7\eps/8 \qand  \re \psi_{n+1}  (u^* z)> \re \phi_n(u^* z)-5\eps/8
\]
for each $n\geq N$. Given any integer $m\geq 1$, we thus find
\begin{align*}
1\geq \re \phi_{N+m}(u^*z)&\geq \re \psi_{N+m}(u^* z)+7\eps/8> \re \phi_{N+m-1}(u^*z)+\eps/4\\
&\geq \ldots \geq \re \phi_{N}(u^*z)+m\eps/4
\end{align*}
which is absurd.
\end{proof}

We now obtain our desired analogue of Mazur's theorem.

\begin{theorem}\label{T:Qsmoothdense}
Let $B$ be a unital separable $\rC^*$-algebra and let $A\subset B$ be a unital  norm-closed subalgebra. Let $A_0\subset A$ be the subset of elements $x$ with $\|x\|=1$ for which the set $R^\Q_x$ is a singleton. Then, $A_0$ contains a $G_\delta$ dense subset of the unit sphere of $A$.
\end{theorem}
\begin{proof}
Denote by $U$ the unit sphere of $A$. 
Since $A$ is separable, so is $U$, and there is a countable subset $\{u_n:n\geq 1\}$ which is dense in $U$. For each $n,m \geq 1$, we set
\[
D_{nm}=\{x\in U:  \re (\phi(u_n^* x)-\psi(u_n^* x))<1/m \text{ for every } \phi,\psi\in N^\Q_x\}.
\]
Let $x\in \bigcap_{n,m=1}^\infty D_{nm}$ and let $\phi,\psi\in N^\Q_x$. Then, $ \re \phi(u^*_n x)=\re \psi(u^*_nx)$   for every $n\geq 1$. By density, this means that $\re\phi(u^*x)=\re\psi(u^*x)$ for every $u\in U$. For each $u\in U$, using that $A$ is a subalgebra we have $xu\in U$ so  $\re \phi(u^*x^*x )=\re\psi(u^*x^* x)$. Observe now that $x^*x$ belongs to the multiplicative domain of $\phi$ and $\psi$ since $\phi,\psi\in N^\Q_x$, so  $\re \phi(u^*)=\re \psi(u^*)$ for every $u\in U$. We conclude that $\re\phi=\re \psi$ on $A$, and hence $\phi=\psi$ on $A$.
In other words,  $x\in A_0$. This shows that $A_0$ contains $\bigcap_{n,m=1}^\infty D_{nm}.$ 
Using Lemma \ref{L:dense} along with  the Baire category theorem, we see that $\bigcap_{n,m=1}^\infty D_{nm}$ is a $G_\delta$ dense subset of $U$, and the proof is complete.
\end{proof}

Although our argument above requires $A$ to be an algebra to conclude that $xu\in U$, it is conceivable that the result itself could hold for unital norm-closed subspaces as well, but we have not been able to adapt the argument to this more general setting.

We now obtain a result illustrating that $\Q(A)$-peak sets are always somewhat plentiful.

\begin{corollary}\label{C:Qpeaknorm}
Let $B$ be a unital separable $\rC^*$-algebra and let $A\subset B$ be a unital norm-closed subalgebra. Let $\P(\Q(A))$ denote the set of pure states $\omega$ on $B$ such that $E_\omega$ contains a non-empty $\Q(A)$-peak set. Then, 
$
\sup\{ \omega(a^*a):\omega\in \P(\Q(A))\}=\|a\|^2
$
for every $a\in A$.
\end{corollary}
\begin{proof}
Let $A_0\subset A$ be defined as in Theorem \ref{T:Qsmoothdense}. For each $x\in A_0$, there always is a pure state $\psi$ on $B$ lying in $N^\Q_x$. By definition of $A_0$, we find $R^\Q_x=\{\psi|_M\}$.  Proposition \ref{P:peakset} implies that $E_\psi$ contains a non-empty $\Q(A)$-peak set, so $\psi\in\P(\Q(A))$ which implies
\[
\sup_{\omega\in \P(\Q(A))}\omega(x^*x)=\|x\|^2.
\]
Since $A_0$ is norm-dense in the unit sphere of $A$ by virtue of Theorem \ref{T:Qsmoothdense}, we conclude that
\[
\sup_{\omega\in \P(\Q(A))}\omega(a^*a)=\|a\|^2, \quad a\in A.
\]
\end{proof}

\subsection{From peak sets to peak states in matrix algebras}

At the time of this writing, we do not know if the statement of Corollary \ref{C:Qpeaknorm} can generally be improved by replacing peak sets with peak states.  We close this section by exhibiting instances of small matrix algebras where this can be done. The key observation is the following.

\begin{lemma}\label{L:E}
Let $n\geq 1$ and  let $M\subset \bM_n$ be a unital subspace. 
Let $x\in M$ with $\|x\|=1$ such that $R^\Q_x$ is a singleton while $N^\Q_x$ contains two distinct pure states $\phi$ and $\psi$. Then, the following statements hold.
\begin{enumerate}[{\rm (i)}]
\item There is a two dimensional subspace $V\subset \bC^n$ with the property that $P_V a|_V=\phi(a) I$ for each $a\in M$.

\item If we put $z=P_{V^\perp}x|_V$, then  $z^*z=(1-|\phi(x)|^2)I$.
\end{enumerate}
\end{lemma}
\begin{proof}
The pure states $\phi$ and $\psi$ are vector states. Thus, there exist two linearly independent unit vectors $\xi,\eta\in \bC^n$ such that
\[
\phi(b)=\langle b\xi,\xi\rangle \qand \psi(b)=\langle b\eta,\eta\rangle
\]
for every $b\in \bM_n$. Let $V\subset \bC^n$ be the subspace spanned by $\xi$ and $\eta$. The assumption that $\phi,\psi\in N^\Q_x$ means that
$
1=\|x\|^2=\phi(x^*x)=\psi(x^*x).
$
When combined with the Cauchy--Schwarz inequality, this implies that $V\subset \ker( x^*x-I)$. In particular, if $v\in V$ is a unit vector then the state $\theta$ on $\bM_n$ defined as
\[
\theta(b)=\langle bv,v\rangle, \quad b\in \bM_n
\]
satisfies $\theta(x^*x)=1$ or $\theta\in N^\Q_x$. Since $R^\Q_x$ is a singleton, this means that $\theta$ and $\phi$ agree on $M$, so
\[
\langle av,v \rangle=\phi(a), \quad a\in M.
\]
Because $v\in V$ is arbitrary, we infer that 
\[
P_V a|_V=\phi(a) I, \quad a\in M
\]
which is (i).  Finally, put $z=P_{V^\perp}x|_V$. We compute using (i) that
\begin{align*}
I=P_{V}x^*x|_V&=P_V x^* P_V x |_V+P_V x^*P_{V^\perp} x|V\\
&=|\phi(x)|^2 I+z^*z
\end{align*}
so (ii) is established.
\end{proof}

In some case, we can now replace $\Q(M)$-peak sets by $\Q(M)$-peak states.

\begin{proposition}\label{P:M23}
Let $n\geq 1$ and  let $M\subset \bM_n$ be a unital subspace. Put $B=\rC^*(M)$.  Assume that either
\begin{enumerate}[{\rm (a)}]
\item $n\leq 2$, or
\item $n=3$ and the only normal matrices in $M$ are scalar multiples of the identity. 
\end{enumerate} 
Let $x\in M$ with $\|x\|=1$ such that $R^\Q_x$ is a singleton and let $\omega\in N^\Q_x$. Then, $\omega$ is a $\Q(M)$-peak state.
\end{proposition}
\begin{proof}
If $M=\bC I$, then $B=\rC^*(M)=\bC I$ and the claim is trivial. Hence, we assume throughout that  $n\geq 1$ and  that $\bC I$ is properly contained in $M$.

 To show that $\omega$ is a $\Q(M)$-peak state, it suffices to fix another pure state $\psi$ on $B$ such that $\psi(x^*x)=1$, and to establish that $\psi=\omega$ \cite[Lemma 2.1]{CTh2023}.

If $\omega$ and $\psi$ are distinct, then we may apply Lemma \ref{L:E} to find a two dimensional subspace $V\subset \bC^n$ such that \[
P_V a|_V=\omega(a)I, \quad a\in M
\]
and $z^*z=(1-|\omega(x)|^2)I$ for $z=P_{V^{\perp}}x|_V$.

If $n=2$, then $V=\bC^n$ and $a=\omega(a)I$ for every $a\in M$, contrary to our standing assumption that $\bC I\neq M$. 

Next, assume $n=3$ and that the only normal matrices in $M$ are scalar multiples of the identity. 
In this case, $\dim V^\perp=1<\dim V$, so $z^*z$ cannot be invertible. This forces $|\omega(x)|=1$, so that $P_V x|_V$ is a unimodular scalar multiple of the identity on $V$. Because $x$ has norm $1$, we infer that $V$ is a reducing subspace for $x$. Therefore, $x$ is a diagonal matrix with respect to the decomposition $\bC^n=V\oplus V^\perp$, and hence it is normal. By assumption, we must have that $x\in \bC I$. It follows that $\S(B)=N^\Q_x$, so $\{\omega|_M\}=R^\Q_x=\S(M)$, which forces $M=\bC I$. Once again, this contradicts our standing assumption.

In both cases, we thus conclude that $\psi=\omega$, so that indeed $\omega$ is a $\Q(M)$-peak state.
\end{proof}

We now arrive at our desired result.

\begin{corollary}\label{C:M23dense}
Let $n\geq 1$ and  let $A\subset \bM_n$ be a unital subalgebra.  Assume that either
\begin{enumerate}[{\rm (a)}]
\item $n\leq 2$, or
\item $n=3$ and the only normal matrices in $A$ are scalar multiples of the identity. 
\end{enumerate} 
Then, 
\[
\sup\{  \omega(a^*a):\omega \text{ is a } \Q(A)\text{-peak state on } B\}=\|a\|^2
\]
for every $a\in A$.
\end{corollary}
\begin{proof}
By Theorem \ref{T:Qsmoothdense},  it suffices to prove that for a fixed $a\in A_0$, there is a $\Q(A)$-peak state $\omega$ with $\omega(a^*a)=1$. For such an element $a$, there always exists  a pure state $\omega$ on $B$ with $\omega(a^*a)=1$. But then Proposition \ref{P:M23} implies that $\omega$ is a $\Q(A)$-peak state, as desired.
\end{proof}

Of course, this result is of limited applicability, given how restrictive the assumptions are.  One non-trivial case where it is applicable  is the subalgebra of upper triangular Toeplitz matrices in $\bM_3$.

Finally, it is natural to wonder if Theorem \ref{T:P0dense} can be similarly improved to replace peak sets by peak points for small matrix algebras. We do not know the answer, but we mention that the results \cite{LLS2014} appear relevant in this regard.

\section{Unique state extensions for peak states}\label{S:peakuep}

We return to our main question, which we recall for convenience.

\vspace{3mm}

\textbf{Main question.} Let $B$ be a unital $\rC^*$-algebra and let $M\subset B$ be a subspace. Let $\omega$ be a pure state on $B$ and let $E_\omega$ denote the set of states on $B$ extending $\omega|_M$. When does $E_\omega$ reduce to the singleton $\{\omega\}$?
\vspace{3mm}

In this section, we thoroughly examine whether the noncommutative analogues of peak points considered hitherto provide an answer.

First, we note that if $\omega$ is an $M$-peak state, then it follows immediately from the definition that $\omega$ is the unique state on $B$ extending $\omega|_M$. Thus, we have the following implication.
\begin{equation}\label{Eq:imp2}
\begin{tikzcd}[arrows=Rightarrow]
\substack{\omega \text{  is}\\ M\text{-peak state}} \arrow[r] &  \substack{E_\omega=\{\omega\}}
\end{tikzcd}
\end{equation}
In some simple situations, we have a converse.

\begin{proposition}\label{P:ueppeaksingle}
Let $B$ be a unital $\rC^*$-algebra. Let $x\in B$ and put $M=\spn\{I,x\}\subset B$. Let $\omega$ be a state on $B$ such that $\omega(x)=\nr(x)$.
Then, the following statements are equivalent.
\begin{enumerate}[{\rm (i)}]
\item $\omega$ is an $M$-peak state.
\item $\omega|_M$ admits a unique extension to a state on $B$.
\end{enumerate}
\end{proposition}
\begin{proof}
(i)$\Rightarrow$(ii): This always holds, as mentioned before the proposition.

(ii)$\Rightarrow$(i): When $x=0$, we have $M=\bC I$ so the assumption implies that there is a unique state on $B$, which is then automatically an $M$-peak state. Assume henceforth that $x\neq 0$. Upon normalizing $x$ if necessary, we may also suppose that $\nr(x)=1$. By assumption, we see that
$\{\phi\in \S(B):\phi(x)=1\}=\{\omega\}$. We then conclude from Lemma \ref{L:halftrick} that indeed $\omega$ is an $M$-peak state.
\end{proof}

Generally speaking however, the converse of \eqref{Eq:imp2} fails, even if $M$ is a unital norm-closed subalgebra of $B$. Our next goal is to construct an example that illustrates this failure.

Let $B$ be a unital $\rC^*$-algebra and let $M\subset B$ be a  subspace.  Define $\A(M)\subset \bM_2(B)$ to be the unital subalgebra consisting of elements of the form $\begin{bmatrix}
\lambda & x\\ 0 & \lambda
\end{bmatrix}$ for $\lambda\in \bC$ and $x\in M$. Further, given a state $\omega$ on $B$, we define a state $\Phi_\omega$ on $\bM_2(B)$ as
\[
\Phi_\omega([b_{ij}])=\langle [\omega(b_{ij})]v,v\rangle_{\bC^2}, \quad [b_{ij}]\in \bM_2(B)
\]
where $v=\frac{1}{2}(1,1)\in \bC^2$. In other words,
\[
\Phi_\omega([b_{ij}])=\frac{1}{2}(\omega(b_{11})+\omega(b_{12})+\omega(b_{21})+\omega(b_{22}))
\]
for each $[b_{ij}]\in \bM_2(B)$. This last expression implies  in particular that, given another state $\omega'$ on $B$, $\Phi_{\omega}=\Phi_{\omega'}$ holds precisely when  $\omega=\omega'$.

\begin{theorem}\label{T:uepcorner}
Let $B$ be a unital $\rC^*$-algebra and let $M\subset B$ be a unital subspace. Let $\omega$ be a state on $B$. Consider the following statements.
\begin{enumerate}[{\rm (i)}]
\item $\omega|_M$ admits a unique extension to a state on $B$.
\item $\Phi_\omega|_{\A(M)}$ admits a unique extension to a state on $\bM_2(B)$.
\item $\omega$ is an $M$-peak state.
\item $\Phi_\omega$ is an $\A(M)$-peak state. 
\end{enumerate}
Then, we have ${\rm (i)} \Leftrightarrow {\rm (ii)}$ and ${\rm (iv)} \Rightarrow {\rm (iii)}$. 
If $B$ is commutative, then ${\rm (iii)} \Leftrightarrow {\rm (iv)}$.
\end{theorem}
\begin{proof}
(ii)$\Rightarrow$(i): Let $\psi$ be a state on $B$ extending $\omega|_M$. Then, it follows that $\Phi_\psi=\Phi_\omega$ on $\A(M)$, so by assumption $\Phi_\psi=\Phi_\omega$, and in turn this implies $\psi=\omega$ as observed before the theorem.

(i)$\Rightarrow$(ii): Let $\Psi$ be a state on $\bM_2(B)$ extending $\Phi_\omega|_{\A(M)}$. Define a linear map $\psi:B\to\bC$ as
\[
\psi(b)=\Psi\left(\begin{bmatrix}
0 & 2b \\ 0 & 0
\end{bmatrix} \right), \quad b\in B.
\]
We claim that $\psi$ is a state on $B$. First, we see that
\begin{align*}
\psi(I)&=\Psi\left(\begin{bmatrix}
0 & 2I \\ 0 & 0
\end{bmatrix} \right)=\Phi_\omega\left(\begin{bmatrix}
0 & 2I \\ 0 & 0
\end{bmatrix} \right)=1.
\end{align*}
Next, to see that $\psi$ is contractive, we let $(\sigma,H,\xi)$ denote the GNS representation of $\Psi$, so that $\sigma:\bM_2(B)\to B(H)$ is a unital $*$-representation with unit cyclic vector $\xi\in H$ satisfying
\[
\Psi([b_{ij}])=\langle \sigma([b_{ij}])\xi,\xi\rangle , \quad [b_{ij}]\in \bM_2(B).
\]
As is well-known, up to unitary equivalence there must be a unital $*$-representation $\pi:B\to B(K)$ such that $H=K\oplus K$ and 
\begin{align*}
\sigma([b_{ij}])=[\pi(b_{ij})], \quad [b_{ij}]\in \bM_2(B).
\end{align*}
Hence
\begin{equation}\label{Eq:GNSPhi}
\Psi([b_{ij}])=\langle ([\pi(b_{ij})])\xi,\xi\rangle , \quad [b_{ij}]\in \bM_2(B).
\end{equation}
According to the decomposition $H=K\oplus K$, we write $\xi=(\eta,\zeta)$ for some vectors $\eta,\zeta\in K$ satisfying $1=\|\eta\|^2+\|\zeta\|^2$. For $b\in B$ we find using \eqref{Eq:GNSPhi} that
\begin{equation}\label{Eq:phiWitt}
\psi(b)=\Psi\left(\begin{bmatrix}
0 & 2b \\ 0 & 0
\end{bmatrix} \right)=\left\langle \begin{bmatrix}
0 & 2\pi(b) \\ 0 & 0
\end{bmatrix}\xi,\xi\right\rangle=2\langle \pi(b)\zeta,\eta\rangle
\end{equation}
so 
\begin{align*}
|\psi(b)|&\leq 2\|\eta\| \|\zeta\| \|b\|\leq (\|\eta\|^2+\|\zeta\|^2)\|b\|=\|b\|.
\end{align*}
This shows that $\psi$ is contractive, and so indeed it is a state. Note also that 
$
1=\psi(I)=2\langle \zeta,\eta\rangle$, so by the Cauchy--Schwarz inequality this implies
\[
1\leq 2\|\zeta\| \|\eta\|\leq  (\|\eta\|^2+\|\zeta\|^2)=1.
\]
Equality must hold throughout, which in turns forces $\eta=\zeta$.

Next, for $x\in M$  we have that  $\begin{bmatrix}
0 & 2x \\ 0 & 0
\end{bmatrix}\in \A(M)$ so that, using that $\Psi$ agrees with $\Phi_\omega$ on $\A(M)$, we find
\[
\psi(x)=\Psi\left(\begin{bmatrix}
0 & 2x \\ 0 & 0
\end{bmatrix} \right)=\Phi_\omega\left(\begin{bmatrix}
0 & 2x \\ 0 & 0
\end{bmatrix} \right)=\omega(x).
\]
We thus conclude that $\psi$ is a state on $B$ extending $\omega|_M$, so by assumption we must have $\omega=\psi$. Finally, for $[b_{ij}]\in \bM_2(B)$ we compute using \eqref{Eq:GNSPhi} and \eqref{Eq:phiWitt} and the fact that $\eta=\zeta$ that
\begin{align*}
\Psi([b_{ij}])&=\langle([\pi(b_{ij})])\xi,\xi\rangle\\
&=\langle \pi(b_{11})\zeta,\eta \rangle+\langle \pi(b_{12})\zeta,\eta \rangle+\langle \pi(b_{21})\zeta,\eta \rangle+\langle \pi(b_{22})\zeta,\eta \rangle\\
&=\frac{1}{2}\psi(b_{11}+b_{12}+b_{21}+b_{22})=\frac{1}{2}\omega(b_{11}+b_{12}+b_{21}+b_{22})=\Phi_\omega([b_{ij}]).
\end{align*}
Therefore, $\Phi_\omega=\Psi$ as desired.

(iv)$\Rightarrow$(iii):  By assumption there is $a=\begin{bmatrix}
\lambda & x \\ 0 & \lambda
\end{bmatrix}\in \A(M)$ with $\lambda\in \bC$ and $x\in M$ such that 
\[
\lambda+\frac{\omega(x)}{2}=\Phi_\omega(a)>|\Phi_\psi(a)|=\left|\lambda+\frac{\psi(x)}{2}\right|
\] for each state $\psi$ on $B$ distinct from $\omega$. Choosing $y=\lambda I +\frac{1}{2}x\in M$, we find $\omega(y)>|\psi(y)|$ for each state $\psi$ on $B$ distinct from $\omega$, so that $\omega$ is an $M$-peak state.

(iii)$\Rightarrow$(iv):  We assume here that $B$ is commutative, and that there is $x\in M$ such that $\omega(x)>|\psi(x)|$ for every state $\psi$ distinct from $\omega$.  Consider $a=\begin{bmatrix}
1 & x \\0 & 1 
\end{bmatrix}\in \A(M)$. 

We first claim that $\Phi_\omega(a)>|\theta(a)|$ for every pure state $\theta$ on $\bM_2(B)$ such that $\theta\neq \Phi_\omega$. Indeed, note that $\bM_2(B)\cong B\otimes_{\min} \bM_2$ \cite[Proposition 12.5]{paulsen2002}, so by \cite[Theorem IV.4.14]{takesaki2002}, we find a pure state $\chi$ on $B$ along with a unit vector $\xi=(\xi_1,\xi_2)\in \bC^2$ such that
\[
\theta([b_{ij}])=\langle [\chi(b_{ij})]\xi,\xi \rangle_{\bC^2}, \quad [b_{ij}]\in \bM_2(B).
\] 
It follows that 
$
\theta(a)=1+\chi(x)\xi_2\ol{\xi_1}
$
whence
\[
\Phi_\omega(a)=1+\frac{1}{2}\omega(x)\geq  1+|\chi(x)| |\xi_1| |\xi_2|\geq |\theta(a)|.
\]
Equality above implies $|\chi(x)|=\omega(x)$, $|\xi_1|=|\xi_2|=\frac{1}{\sqrt{2}}$ and 
\[
|1+\chi(x)\xi_2\ol{\xi_1}|=1+|\chi(x)||\xi_1||\xi_2|.
\]
In turn, this implies $\chi=\omega$, and $\xi_2\ol{\xi_1}\geq 0$, so that $\xi$ is a unimodular scalar multiple of $\frac{1}{\sqrt{2}}(1,1)$,  which then implies $\theta=\Phi_\omega$. This establishes the claim.

Finally, note that the weak-$*$ continuous convex function
$
\Psi\mapsto |\Psi(a)|
$ 
defined on the state space of $\bM_2(B)$ attains its maximum value on the set of pure states, by Bauer's maximum principle \cite[Theorem I.5.3]{alfsen1971}. Hence, $\nr(a)=\Phi_\omega(a)$ by the previous paragraph. The previous paragraph also implies that the weak-$*$ closed face $\{\Psi\in \S(\bM_2(B)):|\Psi(a)|=\nr(a)\}$ has a unique extreme point, and hence must reduce to $\Phi_\omega$ by the Krein---Milman theorem. In other words, $\Phi_\omega$ is an $\A(M)$-peak state.
\end{proof}

\begin{example}\label{E:uepnotpeak}
By \cite[page 43]{phelps2001}, there is a unital, separable, commutative $\rC^*$-algebra $B$, a unital norm-closed subspace $M\subset B$ with $\rC^*(M)=B$, along with a pure state $\omega$ on $B$ for which $\omega|_M$ admits a unique extension to a state on $B$, yet $\omega$ is not an $M$-peak state. Consider the unital norm-closed subalgebra $\A(M)\subset \bM_2(B)$ as above. Then, the state $\Phi_\omega$ on $\bM_2(B)$ is pure (\cite[Theorem IV.4.14]{takesaki2002}). Moreover, by Theorem \ref{T:uepcorner},  the restriction $\Phi_\omega|_{\A(M)}$ admits a unique extension to a state on $\bM_2(B)$ yet $\Phi_\omega$ is not an $\A(M)$-peak state.
\qed
\end{example}
This example shows the following.

\begin{equation}\label{Eq:imp3}
\begin{tikzcd}[arrows=Rightarrow]
\substack{\omega \text{  is}\\ A\text{-peak state}}  &  \substack{E_\omega=\{\omega\}}\arrow[l,"/"{anchor=center,sloped}]
\end{tikzcd}
\end{equation}

As observed in \cite[Example 3.2]{CTh2023}, generally speaking we have that
\begin{equation}\label{Eq:imp4}
\begin{tikzcd}[arrows=Rightarrow]
\substack{E_\omega=\{\omega\}}&  \substack{\omega \text{  is}\\ \Q(A)\text{-peak state}}  \arrow[l,"/"{anchor=center,sloped}] 
\end{tikzcd}
\end{equation}

Our next example will show that
\begin{equation}\label{Eq:imp5}
\begin{tikzcd}[arrows=Rightarrow]
\substack{\omega \text{  is}\\ A\text{-peak state}} \arrow[r,"/"{anchor=center,sloped}] &  \substack{\omega \text{  is}\\ \Q(A)\text{-peak state}} 
\end{tikzcd}
\end{equation}

\begin{example}\label{E:ueppeak2}
Consider the unital subalgebra $A\subset \bM_2$ of upper triangular matrices with constant diagonal.   Let $\xi=\frac{1}{\sqrt{2}}(1,1)\in \bC^2$. Let $\omega$ be the pure state on $\bM_2$ defined as
\[
\omega(t)=\langle t\xi,\xi\rangle, \quad t\in \bM_2.
\]
Let $M=B=\bC$, and let  $\eps$ denote the unique state on $\bC$, which is trivially an $M$-peak state and a $\Q(M)$-peak state. Then, $A=\A(\bC)$ and  $\omega=\Phi_\eps$, so $\omega$ is an $A$-peak state by Theorem \ref{T:uepcorner}.

On the other hand, we claim that $\omega$ is not a $\Q(A)$-peak state. Indeed, assume that there is $a=\begin{bmatrix} \lambda & \mu \\ 0 & \lambda\end{bmatrix}\in A$ such that $\omega(a^*a)>\phi(a^*a)$ for every state $\phi$ on $\bM_2$ distinct from $\omega$. Upon scaling $a$ if necessary, we may assume that $1=\|a\|=\omega(a^*a)$, so that  $a^*a\xi=\xi$. Since 
\[
a^*a=\begin{bmatrix}
|\lambda|^2 & \ol{\lambda}\mu\\ 
\ol{\mu}\lambda & |\lambda|^2+|\mu|^2
\end{bmatrix}
\]
the equality $a^*a\xi=\xi$ is equivalent to
\[
|\lambda|^2+\ol{\lambda}\mu=1=\ol{\mu}\lambda+ |\lambda|^2+|\mu|^2
\]
which implies $2i\im(\ol{\lambda}\mu)=|\mu|^2$. This clearly forces $\mu=0$,  in which case $a=\lambda I$ with $|\lambda|=1$, so that $\phi(a^*a)=1$ for every state $\phi$ on $\bM_2$.
\qed
\end{example}

\subsubsection{Summary of implications}
We close this section by summarizing the connections between the various notions currently at play.

Assume first that $B$ is a unital commutative $\rC^*$-algebra and $M\subset B$ is a unital subspace. Let $\omega$ be a pure state on $B$. The key point is that, because $B$ is commutative, $\omega$ is automatically multiplicative. Then, we have the following diagram of implications.

\[
\begin{tikzcd}[arrows=Rightarrow]
 \substack{\fl_\omega\text{ is}\\ M\text{-peak projection}} \arrow[r,Leftrightarrow] &   \substack{\omega \text{  is}\\ M\text{-peak state}} \arrow[d, shift left=1.5] \arrow[r,Leftrightarrow] & \substack{\omega \text{  is}\\ \Q(M)\text{-peak state}}\\
 & \substack{ E_\omega=\{\omega\}}\arrow[u,"/"{anchor=center,sloped}, shift left=1.5] & 
\end{tikzcd}
\]

Assuming still that $B$ is commutative, if $M$ is replaced by a unital \emph{norm-closed subalgebra} $A$, then we have equivalences throughout: 
\[
\begin{tikzcd}[arrows=Rightarrow]
 \substack{\fl_\omega\text{ is}\\ A\text{-peak projection}} \arrow[r,Leftrightarrow] &   \substack{\omega \text{  is}\\ A\text{-peak state}} \arrow[r,Leftrightarrow] & \substack{\omega \text{  is}\\ \Q(A)\text{-peak state}}\arrow[r,Leftrightarrow]  & \substack{ E_\omega=\{\omega\}}
\end{tikzcd}
\]
Next, we remove the condition that $B$ is commutative. In this general case, for a unital subspace $M\subset B$, we have the following diagram, by virtue of \eqref{Eq:imp1},\eqref{Eq:imp2},\eqref{Eq:imp3},\eqref{Eq:imp4} and \eqref{Eq:imp5}.
\begin{equation}\label{Eq:diagram}
\begin{tikzcd}[arrows=Rightarrow]
 &  \substack{\fl_\omega\text{ is}\\ M\text{-peak projection}} \arrow[dl, shift left=1.5] \arrow[dr,shift left=1.5] &\\
\substack{\omega \text{  is}\\ M\text{-peak state}}\arrow[dr, shift left=1.5] \arrow[rr,"/"{anchor=center,sloped}, shift right=1.5] \arrow[ur,"/"{anchor=center,sloped}, shift left=1.5]& & \substack{\omega \text{  is}\\ \Q(M)\text{-peak state}}\arrow[dl, "/"{anchor=center,sloped}, shift left=1.5]\arrow[ll,"/"{anchor=center,sloped}, shift right=1.5]\arrow[ul,"/"{anchor=center,sloped}, shift left=1.5]\\
 & \substack{ E_\omega=\{\omega\}}\arrow[ul,"/"{anchor=center,sloped}, shift left=1.5]\arrow[ur,"/"{anchor=center,sloped}, shift left=1.5]&
\end{tikzcd}
\end{equation}
Notably, the counterexamples justifying the negated implications above occur for unital norm-closed subalgebras, so that the noncommutative world exhibit completely new intricacies and pathologies. In particular, none of our current  noncommutative analogues of peak points provides a complete characterization of $E_{\omega}=\{\omega\}$, so our main question remains open.

\section{The characterization: pinnacle sets}\label{S:pinndet}

In this section, we prove one of the main results of the paper, which gives a characterization of the unique extension property for states in terms of weaker form of peaking, thereby answering our main question.

\subsection{Preliminaries}

We start our preparation with a computational fact about derivatives of certain exponential functions in non-commutative algebras.

\begin{lemma}\label{L:expderiv}
Let $B$ be a unital $\rC^*$-algebra and let $a\in B$. Let $\psi$ be a state on $B$. 
Define a function $f:\bR\to [0,\infty)$ as
\[
f(t)=\psi(e^{ta*}e^{ta}), \quad t\in \bR.
\]
Then, for every $t\in \bR$ we have
\[
f'(t)=2\psi(e^{ta*}(\re a) e^{ta})
\]
and
\[f''(t)=\psi(e^{ta*}(4(\re a)^2+(a^*a-aa^*)) e^{ta}).
\]
In particular, if  $4(\re a)^2+(a^*a-aa^*)\geq 0$, then $f$ is convex.
\end{lemma}
\begin{proof}\
For distinct numbers $s,t\in \bR$ we compute
\begin{align*}
\frac{1}{t-s}(f(t)-f(s))&=\frac{1}{t-s}\psi\left( e^{sa*}e^{(t-s)a*}e^{(t-s)a}e^{sa}- e^{sa*}e^{sa} \right)\\
&=\psi\left( e^{sa*}\frac{1}{t-s}\left(e^{(t-s)a*}e^{(t-s)a}-I\right)e^{sa}\right).
\end{align*}
Letting $s\to t$, we infer that
\[
f'(t)=\psi(e^{ta*}d_1e^{ta}), \quad t\in \bR
\]
where $d_1\in B$ is the derivative at $t=0$ of the function
$
t\mapsto e^{ta*}e^{ta}.
$
Power series expansions yield that $d_1=2\re a$, so  we obtain
\begin{align*}
f'(t)= 2\psi\left( e^{ta*}\left(\re a\right)e^{ta}\right), \quad t\in\bR.
\end{align*}
Similarly, for $s\neq t$ we now find
\begin{align*}
\frac{1}{t-s}(f'(t)-f'(s))
&= \frac{2}{t-s}\psi\left( e^{sa*}e^{(t-s)a*}(\re a)e^{(t-s)a}e^{sa} - e^{sa*}(\re a) e^{sa}\right)\\
&= 2\psi\left( e^{sa*}\frac{1}{t-s}\left(e^{(t-s)a*}( \re a) e^{(t-s)a}-\re a\right)e^{sa}\right).
\end{align*}
Letting $s\to t$, we infer that
\[
f''(t)=2\psi(e^{ta*}d_2e^{ta}), \quad t\in \bR
\]
where $d_2\in B$ is the derivative at $t=0$ of the function
$
t\mapsto e^{ta*}(\re a)e^{ta}.
$
Power series expansions yield that $d_2=a^*(\re a)+(\re a)a$, so  we obtain
\begin{align*}
f''(t)=2 \psi\left( e^{ta*}\left(a^*(\re a)+(\re a)a\right)e^{ta}\right), \quad t\in \bR.
\end{align*}
Finally, put $b=\re a$ and $c=\im a$. For $t\in \bR$, we compute
\begin{align*}
a^*(\re a)+(\re a)a&=b^2-i cb+b^2+ibc=2b^2+i(bc-cb)
\end{align*}
and
\[
a^*a-aa^*=2i(bc-cb)
\]
whence
\[
f''(t)= \psi\left( e^{ta*}\left(4(\re a)^2+(a^*a-aa^*)\right)e^{ta}\right).
\]
The final statement follows at once from this identity.
\end{proof}

Next, we establish some crucial estimates.

\begin{lemma}\label{L:expbound}
Let $B$ be a unital $\rC^*$-algebra and let $a\in B$. Let $\psi$ be a state on $B$. 
Then, the following statements hold.
\begin{enumerate}[{\rm (i)}]
\item For each $t\geq 0$, we have $\|e^{ta}\|\leq \|e^{\re a}\|^t$.
\item Let $\kappa=6 \|a\|^2 \|e^{\re a}\|^2$. Then, 
\[
| \psi(e^{ta*}e^{ta})-(1+2 (\re\psi(a))t)|\leq \kappa t^2
\]
for every $0\leq t\leq 1$.
\end{enumerate}
\end{lemma}
\begin{proof}
By a well-known inequality (see for instance \cite[Theorem 3.2]{pryde1991}), we have $\|e^{ta}\|\leq \|e^{ \re (ta)}\|$ for each $t\in \bR$. Thus, for each $t \geq 0$, we find $\|e^{ta}\|\leq \|e^{\re a}\|^t$, which is (i).

Next, define a function $f:\bR\to [0,\infty)$ as
\[
f(t)=\psi(e^{ta*}e^{ta}), \quad t\in \bR.
\]
By combining Lemma \ref{L:expderiv} with (i), we see that $f'(0)=2\re\psi(a)$ and
\[
|f''(t)|\leq \kappa, \quad 0\leq t\leq 1.
\]
In turn, by the mean value theorem, given $0< t\leq 1$, there are real numbers $u$ and $v$ with $0< v<u<t$ such that
\[
f(t)=f(0)+f'(u)t
\]
and
\begin{align*}
f(t)&=f(0)+(f'(0)+f''(v)u)t\\
&=1+2(\re \psi (a))t+f''(v)u t.
\end{align*}
Hence
\[
| f(t)-(1+2 (\re\psi(a))t)|\leq \kappa t^2
\]
for every $0\leq t \leq 1$, which establishes (ii). 
\end{proof}

Next, we reformulate the property of a state admitting a unique extension in terms of certain real subspaces of self-adjoint elements.  The following is routine, but we lack an exact reference so we include the proof for the sake of completeness.

\begin{lemma}\label{L:uepreal}
Let $B$  be a unital $\rC^*$-algebra and let $M\subset N\subset B$ be unital subspaces. Let $\omega$ be a state on $B$. Assume that the restriction $\omega|_M$ admits a unique state extension to $N$. Consider the following real subspaces of $B$:
\[
R=\spn_{\bR} \{\re a:a\in M\}, \quad S=\spn_{\bR}\{\re b:b\in N\}.
\]
Then, the restriction $\omega|_R$ admits a unique $\bR$-linear contractive extension to $S$.
\end{lemma}
\begin{proof}
Put $\rho=\omega|_R$ and let $\rho':S\to\bR$ be an $\bR$-linear contractive extension of $\rho$. We may then define $\omega':N\to \bC$ as
\[
 \omega'(b)=\rho'(\re b)-i\rho'(\re (ib)), \quad b\in N.
\]
 It is readily verified that the functional $\omega'$ is complex linear and unital. Next, given $b\in N$ there is a unimodular complex number $\zeta$ such that $|\omega'(b)|=\omega'(\zeta b)$. Thus,
 \[
 |\omega'(b)|=\re \omega'(\zeta b)=\rho' (\re (\zeta b))\leq \|\re (\zeta b)\|\leq \|b\|
 \]
 so that $\omega'$ is contractive, and hence is a state on $N$. In addition, $\omega'$ extends $\omega|_M$ since $\rho'$ extends $\rho$. By assumption, we then infer that $\omega'=\omega|_N$. It follows that
 \[
 \rho'(\re b)=\omega'(\re b)=\omega(\re b),\quad b\in N
 \]
 so indeed $\rho$ has a unique $\bR$-linear contractive extension to $S$.
\end{proof}

\subsection{The main result}
With these preliminaries out of the way, we can prove a technical result that reformulates the unique extension property for a state in terms of a weak form of peaking.

\begin{lemma}\label{L:ueppinnequiv}
Let $B$ be a unital $\rC^*$-algebra and let $A\subset B$ be a unital norm-closed subalgebra. Let $\omega$ be a state on $B$, and let $f\in B$ be a positive contraction such that $\omega(f)=1$.  Then, the following statements are equivalent.
\begin{enumerate}[{\rm (i)}]
\item  $\omega|_A$ has a unique extension to a state on $A+\bC f$.
\item For every $0<\gamma<1$ and $\eps>0$ satisfying $\log(\gamma+\eps)<0$, there is $x\in A$ with $\|x\|\leq 1+\eps$ such that
\[
\Omega(x^*x) >1>\psi(x^*x) 
\]
if  $\psi$ is a state on $B$ with $\psi(f)<\gamma$, and $\Omega$ is a state on $B$ agreeing with $\omega$ on $A$.
\item For every $0<\gamma<1$, there is $x\in A$ such that
\[
\Omega(x^*x) >1>\psi(x^*x) 
\]
if  $\psi$ is a state on $B$ with $\psi(f)<\gamma$, and $\Omega$ is a state on $B$ agreeing with $\omega$ on $A$.
\end{enumerate}
\end{lemma}
\begin{proof}

(ii)$\Rightarrow$(iii): This is trivial.

(iii)$\Rightarrow$(i): Let $\psi$ be a state on $B$ with $\psi(f)\neq \omega(f)$, so that $\psi(f)<1$. Choose $0<\gamma<1$ small enough so that $\psi(f)<\gamma$.  Then, the assumed inequality immediately implies that $\psi$ cannot agree with $\omega$ on $A$. 

(i)$\Rightarrow$(ii): Fix $0<\gamma<1$ and $\eps>0$ satisfying $\log(\gamma+\eps)<0$.
Let $g=\log (f+\eps I)\in B$. Let $\Omega$ be a state on $B$ agreeing with $\omega$ on $A$.  By assumption, we find that $\Omega(f)=\omega(f)=1$. Because $f$ is a positive contraction, it follows from the Schwarz inequality that $f$ lies in the multiplicative domain of $\Omega$, and we see that
\[
\Omega(g)=\log(\Omega(f)+\eps)=\log(1+\eps).
\]
In other words, $\omega|_A$ has a unique extension to a state on $A+\bC g$. By virtue of Lemma \ref{L:uepreal}, we know that the restriction of $\omega$ to the real subspace $\{\re a:a\in A\}$ admits a unique $\bR$-linear extension to a contractive functional  on the real subspace $\{\re a+t g:t\in \bR, a\in A\}$, namely the corresponding restriction of $\omega$.  
The classical Hahn--Banach argument (see for instance \cite[Proposition 6.2]{arveson2011}) then implies that
\[
\log(1+\eps)=\omega(g)=\sup\{\omega(\re a):a\in A, \re a\leq g\}.
\]
Hence, there is $a\in A$ with $\re a\leq g$ such that 
\begin{equation}\label{Eq:omega}
\omega(\re a)\geq (1-\eps)\log(1+\eps).
\end{equation}
By Lemma \ref{L:expbound}, for $t\geq 0$ we see that
\begin{equation}\label{Eq:expnorm}
\|e^{ta}\|\leq \|e^{ \re a}\|^t\leq \|e^g\|^t=\|f+\eps I\|^t\leq (1+\eps)^t.
\end{equation}
Given a state $\phi$ on $B$, we define $f_\phi:\bR\to[0,\infty)$ as
\[
f_\phi(t)=\phi(e^{ta*}e^{ta}), \quad t\in \bR.
\]
By Jensen's inequality, for every state $\phi$ on $B$ we see that
\[
\phi(g)\leq \log(\phi(e^g))=\log( \phi(f)+\eps ).
\]
Denote by $K$ the set of states $\psi$ on $B$ with  $\psi(f)<\gamma$, and by $E$ the set of states $\Omega$ on $B$ agreeing with $\omega$ on $A$. We obtain
\begin{equation}\label{Eq:psire}
\psi(\re a)\leq \psi(g)\leq \log(\psi(f)+\eps)<\log(\gamma+\eps)
\end{equation}
for each $\psi\in K$.
If we let $\kappa=6 \|a\|^2 \|e^{\re a}\|^2$
 then by Lemma \ref{L:expbound} along with \eqref{Eq:omega}, we see that
 \begin{equation}\label{Eq:fomega}
f_\Omega(t)\geq 1+2\omega(\re a)t-\kappa t^2\geq 1+(2(1-\eps)\log(1+\eps)-\kappa t)t
\end{equation}
for each $0\leq t\leq 1$ and each $\Omega\in E$.
Likewise,  for $\psi\in K$,  Lemma \ref{L:expbound} along with \eqref{Eq:psire} gives
\begin{equation}\label{Eq:fpsi}
f_\psi(t)\leq 1+2\psi(\re a)t+\kappa t^2\leq 1+(2\log(\gamma+\eps)+\kappa t)t
\end{equation}
for each $0\leq t\leq 1$.

Next, choose $0<\delta<1$ small enough so that, for $0<t\leq \delta$, we have
\begin{equation}\label{Eq:delta2}
2(1-\eps)\log(1+\eps)-\kappa t>(1-\eps)\log(1+\eps)
 \end{equation}
and
\begin{equation}\label{Eq:delta3}
2\log(\gamma+\eps)+\kappa t<\log(\gamma+\eps).
 \end{equation}
 Note that this last inequality is achievable since $\log(\gamma+\eps)<0$.
By virtue of \eqref{Eq:fomega},\eqref{Eq:fpsi},\eqref{Eq:delta2} and \eqref{Eq:delta3}, we find
\begin{align*}
f_\Omega(t)\geq 1+(1-\eps)\log(1+\eps)t>1>1+\log(\gamma+\eps)t \geq f_\psi(t)
\end{align*}
for each $0<t\leq\delta$, each $\psi\in K$ and each $\Omega\in E$. Finally, we may choose $x=e^{\delta a}$, which satisfies $\|x\|\leq (1+\eps)^\delta<(1+\eps)$ by \eqref{Eq:expnorm}.
\end{proof}

We emphasize that a key role in the previous proof is played by \eqref{Eq:omega}, which is where the assumption that $\omega|_A$ admits a unique extension comes into play.

Before proceeding with our main result, we need one more technical fact.

\begin{lemma}\label{L:f}
Let $B$ be a unital $\rC^*$-algebra and let $\omega$ be a pure state on $B$. Let $K$ be a weak-$*$ compact subset of $\S(B)$ not containing $\omega$.  Then, there exist a number $0< \gamma<1$ and a positive contraction $f\in B$ with $\omega(f)=1$ such that
\[
K\subset \{\psi\in \S(B):\psi(f)<\gamma\}.
\]
\end{lemma}
\begin{proof}
Since $\omega\notin K$, by \cite[Lemma 2.2]{clouatre2018lochyp} we see that $\psi(\fl_\omega)<1$ for every $\psi\in K$. In other words, 
\[
K\subset \{\psi\in \S(B):\psi(\fl_\omega)<1\}.
\]
Invoke \cite[Lemma 2.6]{CTh2023} to see that $\fl_\omega\in B^{**}$ is a closed projection. Hence, there is an decreasing net $(f_\lambda)$ of positive contractions in $B$ that converges to $\fl_\omega$ in the weak-$*$ topology of $B^{**}$. Consequently, we may write
\begin{align*}
\{\psi\in \S(B):\psi(\fl_\omega)<1\}&= \bigcup_{m=1}^\infty \{\psi\in \S(B):\psi(\fl_\omega)<1-1/m\}\\
&= \bigcup_{m=1}^\infty\bigcup_{\lambda}\{\psi\in \S(B):\psi(f_\lambda)<1-1/m\}.
\end{align*}
By compactness, there are finitely many pairs of indices $(\lambda_1,m_1),\ldots,(\lambda_r,m_r)$ such that
\[
K\subset \bigcup_{i=1}^r \{\psi\in \S(B):\psi(f_{\lambda_i})<1-1/{m_i}\}.
\]
Put $M=\max\{m_1,\ldots,m_r\}$ and choose an index $\beta$ with  $\beta\geq \lambda_i$ for every $1\leq i\leq r $. Set $f=f_\beta$ and $\gamma=1-1/M$. Using that $f\leq f_{\lambda_i}$ and $\gamma\geq 1-1/{m_i}$ for every $1\leq i\leq r$, we find
\[
K\subset  \bigcup_{i=1}^r \{\psi\in \S(B):\psi(f_{\lambda_i})<1-1/{m_i}\}\subset \{\psi\in \S(B):\psi(f)<\gamma\}.
\]
Finally, the inequalities  $\fl_\omega\leq f\leq I$ imply $\omega(f)=1$.
\end{proof}

 Let $B$ be a unital $\rC^*$-algebra and let $M\subset B$ be a unital subspace. Let $\omega$ be a state on $B$ and let $E_{\omega}$ the set of states $\phi$ on $B$ such that $\phi|_M=\omega|_M$. Given $\eps>0$, we say that $E_\omega$ is an \emph{$(M,\eps)$-pinnacle set}  if,  given any compact subset $K\subset \S(B)\setminus\{\omega\}$, we can find $x\in M$ with  $\|x\|\leq 1+\eps$ such that $\Omega(x^*x)>1>\psi(x^*x)$ for every $\Omega\in E_\omega$ and $\psi\in  K$. 

We can now prove one of the main results of the paper, which answers our main question for subalgebras by characterizing the unique extension property.

\begin{theorem}\label{T:ueppeak}
Let $B$ be a unital $\rC^*$-algebra and let $A\subset B$ be a unital norm-closed subalgebra. Let $\omega$ be a pure state on $B$. Then, the following statements are equivalent.
\begin{enumerate}[{\rm (i)}]
\item The restriction $\omega|_A$ admits a unique state extension to $B$.
\item $E_\omega$ is an $(A,\eps)$-pinnacle set for every $\eps>0$.
\item $E_\omega$ is an $(A,\eps)$-pinnacle set for some $\eps>0$.
\item For each compact subset $K\subset \S(B)$ not containing $\omega$,  there is $x\in A$ such that
\[
\Omega(x^*x) >1>\psi(x^*x) 
\]
for $\psi\in K$ and $\Omega\in E_\omega$.
\end{enumerate}
\end{theorem}
\begin{proof}
(i)$\Rightarrow$(ii): Let  $\eps>0$ and $K\subset \S(B)$ be a compact subset not containing $\omega$. By Lemma \ref{L:f}, there exist a number $0< \gamma<1$ and a positive contraction $f\in B$ with $\omega(f)=1$ such that
\[
K\subset \{\psi\in \S(B):\psi(f)<\gamma\}.
\]
Applying Lemma \ref{L:ueppinnequiv}, we find $x\in A$ such that $\|x\|\leq 1+\eps$ and $\Omega(x^*x)>1>\psi(x^*x)$ for each $\Omega\in E_\omega$ and $\psi\in K$.

(ii)$\Rightarrow$(iii)$\Rightarrow$(iv): This is trivial.

(iv)$\Rightarrow$(i): Let $\Omega$ be a state on $B$ distinct from $\omega$. Applying the assumption with $K=\{\Omega\}$, we see that $\Omega$ cannot agree with $\omega$ on $A$.
\end{proof}

Next, we  show how to answer our main question in the general case of a subspace. Recall that, given a unital subspace $M\subset B$ and a state $\omega$ on $B$, the algebra $\A(M)\subset \bM_2(B)$ and the state $\Phi_\omega$ were defined before Theorem \ref{T:uepcorner}.

\begin{corollary}\label{C:ueppeakspace}
Let $B$ be a unital $\rC^*$-algebra and let $M\subset B$ be a unital norm-closed subspace. Let $\omega$ be a pure state on $B$. Then, the following statements are equivalent.
\begin{enumerate}[{\rm (i)}]
\item The restriction $\omega|_M$ admits a unique state extension to $B$.
\item The set $\{\Psi\in \S(\bM_2(B)):\Psi|_{\A(M)}=\Phi_\omega|_{\A(M)}\}$ is an $(\A(M),\eps)$-pinnacle set for every $\eps>0$.
\item The set $\{\Psi\in \S(\bM_2(B)):\Psi|_{\A(M)}=\Phi_\omega|_{\A(M)}\}$ is an $(\A(M),\eps)$-pinnacle set for some $\eps>0$.
\end{enumerate}
\end{corollary}
\begin{proof}
Simply combine Theorems \ref{T:uepcorner} and \ref{T:ueppeak}.
\end{proof}

\subsection{Bishop's criterion revisited: pinnacle states}

Our definition of pinnacle set is directly inspired by Bishop's criterion discussed in the introduction. This is best seen through the following weaker variant of the definition. 

Let $B$ be a unital $\rC^*$-algebra and let $M\subset B$ be a unital subspace. Let $\omega$ be a state on $B$ and let $\eps>0$. We say that $\omega$ is an \emph{$(M,\eps)$-pinnacle state} if,  given any compact subset $K\subset \S(B)\setminus\{\omega\}$, we can find $x\in M$ with $\|x\|\leq 1+\eps$ such that $\omega(x^*x)>1>\psi(x^*x)$ for every $\psi\in  K$. Clearly, we may shrink the element $x$ and  assume that $\omega(x^*x)=1$.

We now note that this condition is trivially satisfied if $\omega$ is either a $\Q(M)$-peak state or if $E_\omega$ is an $(M,\eps)$-pinnacle set. Combining this with Theorem \ref{T:ueppeak}, we can refine \eqref{Eq:diagram} as follows.

\begin{equation}\label{Eq:diagrampinn}
\begin{tikzcd}[arrows=Rightarrow]
 &  \substack{\fl_\omega\text{  is}\\ M\text{-peak projection}} \arrow[dl, shift left=1.5] \arrow[dr,shift left=1.5] &\\
\substack{\omega \text{  is}\\ M\text{-peak state}}\arrow[dr, shift left=1.5] \arrow[rr,"/"{anchor=center,sloped}, shift right=1.5] \arrow[ur,"/"{anchor=center,sloped}, shift left=1.5]& & \substack{\omega \text{  is}\\ \Q(M)\text{-peak state}}\arrow[d, shift left=1.5]\arrow[dl, "/"{anchor=center,sloped}, shift left=1.5]\arrow[ll,"/"{anchor=center,sloped}, shift right=1.5]\arrow[ul,"/"{anchor=center,sloped}, shift left=1.5]\\
 & \substack{ E_\omega=\{\omega\}}\arrow[ul,"/"{anchor=center,sloped}, shift left=1.5]\arrow[ur,"/"{anchor=center,sloped}, shift left=1.5]&\substack{\omega \text{  is}\\ (M,\eps)\text{-pinnacle}}\arrow[dl, ,"/"{anchor=center,sloped}, shift left=1.5]\arrow[u,"/"{anchor=center,sloped}, shift left=1.5] \\
 &\substack{E_\omega \text{  is}\\ (M,\eps)\text{-pinnacle}}\arrow[u, Leftrightarrow] \arrow[ur, shift left=1.5]& 
\end{tikzcd}
\end{equation}

This diagram suggests that pinnacle states are the weakest of all notions considered hitherto. Nevertheless, as we shall see, some remnants of rigidity can still be observed for these states. The key point of the argument only requires the following property. 

We say that the state $\omega$ is an \emph{$M$-detectable state} if there is a bounded net $(x_\lambda)$ in $M$ such that $\omega(x_\lambda^*x_\lambda)=1$ for every $\lambda$, while for every state $\psi$ distinct from $\omega$, there is $\mu$ such that $\psi(x_\lambda^*x_\lambda)<1$ if $\lambda\geq \mu$. In this case, we say that the net $(x_\lambda)$ \emph{detects} the state $\omega$.

We establish some basic facts about these states.

\begin{lemma}\label{L:detpure}
Let $B$ be a  unital $\rC^*$-algebra and let $M\subset B$ be a unital subspace. States on $B$ that are $M$-detectable must be pure.
\end{lemma}
\begin{proof}
Let $\omega$ be an $M$-detectable state on $B$, which is detected by a net $(x_\lambda)$ in $M$.  Let $\psi_1,\psi_2$ be two states on $B$. If $\psi_1\neq \omega$ and $\psi_2\neq \omega$, then it easily follows from the definition that  there is an index $\lambda$ such that $\psi_1(x_\lambda^*x_\lambda)<1$ and $\psi_2(x_\lambda^*x_\lambda)<1$, while $\omega(x_\lambda^*x_\lambda)=1$. This implies that $\omega\neq (\psi_1+\psi_2)/2$, so that $\omega$ is pure.
\end{proof}

\begin{proposition}\label{P:pinrdet}
Let $B$ be a unital $\rC^*$-algebra and let $M\subset B$ be a unital subspace. Let $\eps>0$. Then, $(M,\eps)$-pinnacle states on $B$ 
are necessarily $M$-detectable. 
Furthermore, when $B$ is separable, $(M,\eps)$-pinnacle states are detected by sequences in $M$.
\end{proposition}
\begin{proof}
Let $\omega$ be an $(M,\eps)$-pinnacle state on $B$. 
Let $\Lambda$ denote the net of compact subsets $K$ of $\S(B)$ not containing $\omega$, ordered by inclusion. By definition, there is  a net $(a_\lambda)_{\lambda\in \Lambda}$ in $M$ with $\|a_\lambda\|\leq 1+\eps$ and $\omega(a_\lambda^*a_\lambda)=1$ for each $\lambda$, while
$
 \phi(a_\lambda^*a_\lambda)<1
$ for $\phi\in \lambda$. Clearly, given a state $\psi$ on $B$ distinct from $\omega$, we have $\{\psi\}\in \Lambda$ so that $\psi(a_\lambda^*a_\lambda)<1$ if $\lambda\geq \{\psi\}$. We conclude that $\omega$ is indeed $M$-detectable.

When $B$ is separable, the weak-$*$ topology on $\S(B)$ is metrizable and compact, so that there is an increasing sequence $(K_n)$ of compact subsets $\S(B)$ such that $\S(B)\setminus\{\omega\}=\bigcup_{n=1}^\infty K_n$. Thus, the sequence $(a_{K_n})$ detects $\omega$.
\end{proof}

We now recall some terminology. Assume that $M$ generates $B$, i.e. $B=\rC^*(M)$. Let $\pi:B\to B(H)$ be a unital  $*$-representation. Following \cite{arveson1969}, we say that $\pi$ has the \emph{unique extension property} with respect to $M$ if, given any unital completely positive map $\psi:B\to B(H)$ agreeing with $\pi$ on $M$, we must have that $\pi$ and $\psi$ agree everywhere on $B$. If, in addition, $\pi$ is irreducible, then we say that $\pi$ is a \emph{boundary representation} for $M$.
\begin{theorem}\label{T:detbdry}
Let $B$ be a unital $\rC^*$-algebra and let $M\subset B$ be a unital subspace. Let $\omega$ be a state on $B$ which is $M$-detectable. Let $\sigma:B\to B(H)$ denote the GNS representation of $\omega$. Then, the following statements hold.
\begin{enumerate}[{\rm (i)}]
\item Let $\theta:B\to B(H)$ be a unital completely positive map that agrees with $\sigma$ on $M$. Then, $\sigma$ is a subrepresentation of any Stinespring representation of $\theta$.

\item $\sigma$ is a boundary representation for $M$.
\end{enumerate}
\end{theorem}
\begin{proof}
(i) There is a unit vector $\xi\in H$ that is cyclic for $\sigma$ and such that
\[
\omega(b)=\langle \sigma(b)\xi,\xi\rangle, \quad b\in B.
\]
Let $\pi:B\to B(K)$ be a Stinespring representation of $\theta$, so that $\pi$ is a unital $*$-representation and  there is an isometry $V:H\to K$ satisfying 
\[
\theta(b)=V^*\pi(b)V,\quad b\in B.
\]
Define 
a state $\phi:B\to \bC$ as
\[
\phi(b)=\langle \theta(b)\xi,\xi\rangle, \quad b\in B.
\]
Note that we also have
\[
\phi(b)=\langle \pi(b)V\xi,V\xi\rangle,\quad b\in B.
\]
Using the Schwarz inequality, for each $a\in M$ we find 
\[
\theta(a^*a)\geq \theta(a)^*\theta(a)=\sigma(a)^*\sigma(a)=\sigma(a^*a).
\]
In particular, if $(a_\lambda)$ is a net in $M$ detecting $\omega$, then
\[
\phi(a_\lambda^*a_\lambda)\geq \omega(a_\lambda^*a_\lambda)=1
\]
for every $\lambda$, whence $\phi=\omega$. Therefore
\[
\omega(b)=\langle \pi(b)V\xi,V\xi\rangle, \quad b\in B.
\]
It follows that $\sigma$ is unitarily equivalent to the restriction of $\pi$ to the reducing subspace $\ol{\pi(B)V\xi}$. 

(ii)  First note that $\sigma$ is irreducible, since $\omega$ is pure by Lemma \ref{L:detpure}. Apply \cite[Theorem 1.2]{dritschel2005} to find a Hilbert space $K$ containing $H$ and a unital $*$-representation $\pi:B\to B(K)$ with the unique extension property with respect to $M$ such that
\[
\sigma(a)=P_H \pi(a)|_H, \quad a\in M.
\]
Define $\theta:B\to B(H)$ as
\[
\theta(b)=P_H \pi(b)|_H, \quad b\in B.
\]
By part (i), we infer that $\sigma$ is  unitarily equivalent to the restriction of $\pi$ to a reducing subspace. Since this restriction also has the unique extension property with respect to $M$ (see for instance \cite[Lemma 2.8]{CTh2022}), so does $\sigma$.
\end{proof}

As a consequence, we can now show that $A$-detectable states are precisely the usual peak points for a subalgebra $A$ of continuous functions.

\begin{corollary}\label{C:pinncomm}
Let $B$ be a unital, commutative $\rC^*$-algebra and let $A\subset B$ be a unital norm-closed subalgebra such that $B=\rC^*(A)$. Let $\omega$ be a state on $B$ which is $A$-detectable. Then, $\fl_\omega$ is an $A$-peak projection and $\omega$ must be given by evaluation at an $A$-peak point.
\end{corollary}
\begin{proof}
By Theorem \ref{T:detbdry}, we see that the GNS representation of $\omega$ is a boundary representation for $A$. Since $B$ is commutative and $\omega$ is pure, we infer that $\omega$ is a character with the unique extension property with respect to $A$. By \cite[Theorem 6.5]{BdL1959}, we find $a\in A$ with $\|a\|=1$ such that $\omega(a)=\omega(a^*a)=1	$ and $\psi(a^*a)<1$ for every pure state on $B$ with $\psi\neq \omega$.  Hence, $\omega$ is evaluation at an $A$-peak point and $\fl_\omega$ is an $A$-peak projection by Proposition \ref{P:peakprojnonorthog}.
\end{proof}

\section{Pinnacle representations}\label{S:rep}

The notions of pinnacle set and pinnacle state from the previous section, while inspired by Bishop's criterion from the introduction, do not provide perfect analogues of it. In this final section, we aim to fix this flaw by replacing states with representations.

Let $B$ be a unital $\rC^*$-algebra. We let $\spec(B)$ denote the \emph{spectrum} of $B$, that is the set of unitary equivalence classes of irreducible $*$-representations of $B$, equipped with its usual topology (see \cite[Chapter 3]{dixmier1977} for details).

 For our  purposes, we require a technical assumption. Let $X\subset B$ be a subset. We say that a state $\psi$ on $B$ is  \emph{$X$-convex} if, for every $a\in X$, the function
\[
t\mapsto \psi(e^{ta*}e^{ta}), \quad t\geq 0
\]
is convex. This condition is easily seen to be satisfied if $\psi$ is multiplicative. We will see more examples below.

\begin{theorem}\label{T:convex}
Let $B$ be a unital $\rC^*$-algebra and let $A\subset B$ be a unital norm-closed subalgebra. Let $\omega$ be a pure state on $B$. Assume that $\omega$ is $A$-convex, and that $\omega|_A$ admits a unique extension to a state on $B$. Let $K\subset \spec(B)$ be a closed subset not containing the class of the GNS representation of $\omega$. Then, for every $\eps>0$ there is $a\in A$ with $\|a\|\leq 1+\eps$ such that $\omega(a^*a)=1$, while $\|\pi(a)\|<\eps$ for every irreducible $*$-representation $\pi$ whose class lies in $K$.
\end{theorem}
\begin{proof}
Throughout the proof, we denote by $\sigma:B\to B(H)$ the GNS representation of $\omega$, with distinguished unit cyclic vector $\xi\in H$ satisfying
\[
\omega(b)=\langle \sigma(b)\xi,\xi\rangle, \quad b\in B.
\]
By definition of the topology on $\spec(B)$, we see that $\sigma$ does not vanish on the ideal $J=\bigcap_{[\pi]\in K}\ker\pi$. The restriction of $\sigma$ to $J$ is thus still irreducible \cite[Lemma 2.11.13]{dixmier1977}. By the Kadison transitivity theorem, there is $0\leq f\leq I$ in $J$ such that $\sigma(f)\xi=\xi$, that is $\omega(f)=1$. Note in particular that $f$ lies in the multiplicative domain of $\omega$.
Letting $g=\log(f+\eps I)\in B$, we find $\omega(g)=\log(1+\eps)$.

By virtue of Lemma \ref{L:uepreal}, we see that the restriction of $\omega$ to the real subspace $\{\re b:b\in A\}$ admits a unique $\bR$-linear extension to a contractive functional on the real subspace of self-adjoint elements in $B$. The classical Hahn--Banach argument (see for instance \cite[Proposition 6.2]{arveson2011}) then implies that
\[
\omega(g)=\sup\{\omega(\re b):b\in A, \re b\leq g\}.
\]
Let $\eps>0$. Then, there is $b\in A$ with $\re b\leq g$ such that 
\[
\omega(\re b)\geq \omega(g)-\eps=\log(1+\eps)-\eps\geq -\eps.
\]
Put $a_0=e^b$. By Lemma \ref{L:expbound}(i), for any $*$-representation $\rho$ of $B$ we see that 
\[
\|\rho(a_0)\|\leq \|e^{ \re \rho(b)}\|\leq \|e^{\rho(g)}\|\leq  \|\rho(f)+\eps I\|.
\]
This implies that $\|a_0\|\leq 1+\eps$ and that $\|\pi(a_0)\|\leq \eps$ if $[\pi]\in K$, since $f\in J$. 

Next, consider the function $f:\bR\to [0,\infty)$ defined as
\[
f(t)=\omega(e^{tb*}e^{tb}), \quad t\in \bR.
\]
Using that $\omega$ is $A$-convex along with Lemma \ref{L:expderiv}, we see that 
\[
f'(t)\geq f'(0)=2\omega( \re b), \quad t\geq 0.
\]
The mean value theorem then implies
\[
\omega(a_0^*a_0)=f(1) \geq  1+2\omega(\re b)\geq 1-2\eps.
\]
Finally, put $a=\frac{1}{\sqrt{\omega(a^*_0a^*)}}a_0$ so that $\omega(a^*a)=1$. We find
\[
\|a\|\leq \frac{1+\eps}{\sqrt{1-2\eps}}
\]
and
\[
\|\pi(a)\|\leq\frac{ \eps}{\sqrt{1-2\eps}}, \quad [\pi]\in K.
\]
This readily implies the desired conclusion.
\end{proof}

To put this result in proper perspective, we next turn to a discussion of the $X$-convexity condition.
Below, we identify a class of subalgebras where every state has the required convexity property. The basic observation is as follows.

\begin{lemma}\label{L:convhypo}
Let $B$ be a unital commutative $\rC^*$-algebra and let $X\subset B$ be a subset. Let $H$ be a Hilbert space. Let $\rho:B\to B(H)$ be a unital completely positive map with the property that
\[
\rho(a^*a)=\rho(a)^*\rho(a), \quad a\in X.
\]
Then, any state on $B(H)$ is $\rho(X)$-convex.
\end{lemma}
\begin{proof}
For $a\in X$, we see that
\[
\rho(a)^*\rho(a)=\rho(a^*a)=\rho(aa^*)
\]
since $B$ is commutative. Then, the Schwarz inequality implies that
\[
\rho(a)^*\rho(a)\geq \rho(a)\rho(a^*)
\]
whence
\[
4(\re \rho(a))^2+(\rho(a)^*\rho(a)-\rho(a)\rho(a)^*)\geq 0
\]
for $a\in X$. Lemma \ref{L:expderiv} then implies that every state on $B(H)$ is $\rho(X)$-convex. 
\end{proof}

Next, we identify a natural class of algebras arising from such a map $\rho$.

\begin{proposition}\label{P:convspherical}
Let $B$ be a unital $\rC^*$-algebra  and let $b_1,\ldots,b_n\in B$ be commuting elements satisfying $\sum_{k=1}^n b_k^* b_k=I$. Let $A\subset B$ denote the unital norm-closed subalgebra generated by $\{b_1,\ldots,b_n\}$. Then, any state on $B$ is $A$-convex.
\end{proposition}
\begin{proof}
We may assume that there is a Hilbert space $F$ such that $B\subset B(F)$. 
It follows from \cite[Proposition 2]{athavale1990} that there is a Hilbert space $H$ containing $F$ and commuting normal operators $u_1,\ldots u_n \in B(H)$ with $\sum_{k=1}^n u_k^* u_k=I$ such that $F$ is invariant for each $u_k$ and
\[
b_k=u_k|_F, \quad 1\leq k\leq n.
\]
Define a unital completely positive map $\rho:\rC^*(u_1,\ldots,u_n)\to B(F)$ as
\[
\rho(t)=P_F t|_F, \quad t\in \rC^*(u_1,\ldots,u_n).
\]
Then, $\rho(u_k)=b_k$ for each $1\leq k\leq n$, and the fact that $F$ is invariant for $u_1,\ldots,u_n$ readily implies that
\[
\rho(c^*c)=\rho(c)^*\rho(c)
\]
for every $c$ in the unital algebra generated by $u_1,\ldots,u_n$.  Moreover, $\rC^*(u_1,\ldots,u_n)$ is commutative by Fuglede's theorem. An application of Lemma \ref{L:convhypo} completes the proof.
\end{proof}

We close the paper with a discussion of Theorem \ref{T:convex}. The result appears very promising at first glance, as its conclusion is a rather close analogue of Bishop's criterion from the introduction. Recall now that it is shown in \cite[Theorem 6.5]{BdL1959} that Bishop's criterion  is equivalent to a genuine peaking condition for subalgebras of $\rC(X)$, when the topology on $X$ is metrizable (or at least not too pathological).  It thus stands to reason that one could adapt the argument of \cite{BdL1959} and use Theorem \ref{T:convex} to prove, for instance, that the GNS representation of $\omega$ is an $A$-peak representation in the sense of \cite{arveson2011}. 

There are some difficulties with  this approach, stemming from the fact that the topology on the spectrum of a general $\rC^*$-algebra need not be Hausdorff. In fact, the class of unital $\rC^*$-algebras with Hausdorff spectra is rather restrictive: such a $\rC^*$-algebra must  be liminal, in the sense that it only admits finite-dimensional $*$-representations \cite[Theorem 1]{BD1975}.

\bibliography{purestatesnsa}
\bibliographystyle{plain}

\end{document}